\documentclass[12pt, reqno]{amsart}
\usepackage{amssymb,amsthm,amsfonts,amsmath, amscd}

\usepackage[hypertex]{hyperref}
\usepackage{mathrsfs}

\newcommand\Y{\mathbb Y}
\newcommand\Z{\mathbb Z}

\newcommand\R{\mathbb R}

\newcommand\GT{{\mathbb{GT}}}

\newcommand\YB{\mathbb{YB}}
\newcommand\E{\mathbb E}

\newcommand\al{\alpha}
\newcommand\be{\beta}
\newcommand\ga{\gamma}
\newcommand\Ga{\Gamma}
\newcommand\de{\delta}
\newcommand\De{\Delta}
\newcommand\ka{\varkappa}
\newcommand\La{\Lambda}
\newcommand\la{\lambda}
\newcommand\si{\sigma}
\newcommand\epsi{\varepsilon}
\newcommand\om{\omega}
\newcommand\Om{\Omega}

\newcommand\wt{\widetilde}

\newcommand\const{\operatorname{const}}

\newcommand\Conf{\operatorname{Conf}}

\newcommand\Sym{\operatorname{Sym}}

\newcommand\Fun{{\operatorname{Fun}}}
\newcommand\dom{\operatorname{Dom}}

\newcommand\MM{\mathfrak M}
\newcommand\LL{\mathfrak L}
\newcommand\X{\mathfrak X}

\newcommand\down{{\downarrow}}
\newcommand\pd{\partial}
\newcommand\FS{F\!S}

\newcommand\M{\mathcal M}
\newcommand\F{\mathcal F}

\newcommand\LLL{\mathscr L}

\newcommand\LB{{}^{\mathbb B}\!\La}
\newcommand\LY{{}^{\Y}\!\La}
\newcommand\LYB{{}^{\YB}\!\La}

\newcommand\z{{(z,z')}}
\newcommand\Dc{D^{(c)}}
\newcommand\Qc{Q^{(c)}}
\newcommand\Qz{Q^{(z,z')}_r}
\newcommand\Tc{T^{(c)}}
\newcommand\Ac{A^{(c)}}
\newcommand\Xc{X^{(c)}}

\newcommand\Tz{T^{(z,z')}}
\newcommand\Az{A^{(z,z')}}

\newcommand\Dz{\mathfrak D^{(z,z')}}

\newcommand\Meix{\mathsf M}
\newcommand\Lag{\mathsf L}

\newtheorem{theorem}{Theorem}[section]
\newtheorem{proposition}[theorem] {Proposition}

\newtheorem{corollary}[theorem]{Corollary}

\theoremstyle{definition}
\newtheorem{definition}[theorem]{Definition}
\newtheorem{remark}[theorem]{Remark}
\newtheorem{example}[theorem]{Example}
\newtheorem{condition}[theorem]{Condition}

\numberwithin{equation}{section}

\begin{document}

\title[Markov dynamics on the Thoma cone]{Markov dynamics on the Thoma cone:
a model \\ of time-dependent determinantal processes\\ with infinitely many
particles}

\author{Alexei Borodin}
\address{Alexei Borodin:
\newline\indent Department of Mathematics, MIT, Cambridge, MA, USA;
\newline\indent Institute for Information Transmission Problems, Moscow, Russia}
\email{borodin@math.mit.edu}

\author{Grigori Olshanski}
\address{Grigori Olshanski:
\newline\indent Institute for Information Transmission Problems, Moscow, Russia;
\newline\indent Independent University of Moscow, Russia;
\newline\indent National Research University Higher School of Economics, Moscow, Russia}

\email{olsh2007@gmail.com}

\date{}

\begin{abstract}
The Thoma cone is an infinite-dimensional locally compact space, which is
closely related to the space of extremal characters of the infinite symmetric
group $S_\infty$. In another context, the Thoma cone appears as the set of
parameters for totally positive, upper triangular Toeplitz matrices of infinite
size.

The purpose of the paper is to construct a family $\{X^\z\}$ of continuous time
Markov processes on the Thoma cone, depending on two continuous parameters $z$
and $z'$. Our construction largely exploits specific properties of the Thoma
cone related to its representation-theoretic origin, although we do not use
representations directly. On the other hand, we were inspired by analogies with
random matrix theory coming from models of Markov dynamics related to
orthogonal polynomial ensembles.

We show that processes $X^\z$ possess a number of nice properties, namely: (1)
every $X^\z$ is a Feller process; (2) the infinitesimal generator of $X^\z$,
its spectrum, and the eigenfunctions admit an explicit description; (3) in the
equilibrium regime, the finite-dimensional distributions of $X^\z$ can be
interpreted as (the laws of) infinite-particle systems with determinantal
correlations;  (4) the corresponding time-dependent correlation kernel admits
an explicit expression, and its structure is similar to that of time-dependent
correlation kernels appearing in random matrix theory.

\end{abstract}

\maketitle

\tableofcontents

\section{Introduction}\label{sect1}

The first two subsections of the introduction contain short preliminary remarks
and a few necessary definitions. Next we state the main results of the paper,
Theorems \ref{thm1.A} and \ref{thm1.B}. Then we describe the method of proof
and make a comparison with some related works.

\subsection{Preliminaries: Markov processes related to orthogonal polynomials}

It is well known that for each family of classical orthogonal polynomials
$p_0,p_1,p_2,\dots$, there exists a second order differential operator $D$,
which preserves the space of polynomials and is diagonalized in the basis
$\{p_n\}$:
$$
Dp_n=m_np_n, \qquad n=0,1,2,\dots,
$$
where $0=m_0>m_1>m_2>\dots$ are the eigenvalues. Let $W(x)$ be the weight
function of $\{p_n\}$ and $\operatorname{supp}W$ be its support. Operator $D$
determines a diffusion Markov process $X$ on $\operatorname{supp}W$ with
$W(x)dx$ being a symmetrizing measure, hence also a stationary distribution.

All these objects, family $\{p_n\}$, operator $D$, and Markov process $X$, have
multidimensional analogs:

Namely, fix $N=2,3,\dots$. {}From $\{p_n\}$ on can construct a family of
symmetric polynomials in $N$ variables indexed by partitions $\nu$ of length at
most $N$, as follows:
$$
p_\nu(x_1,\dots,x_N):=\frac{\det[p_{\nu_i+N-i}(x_j)]}{V(x_1,\dots,x_N)},
$$
where the determinant in the numerator is of order $N$ and
$$
V(x_1,\dots,x_N):=\prod_{1\le i<j\le N}(x_i-x_j).
$$
These polynomials form a basis in the space of symmetric polynomials. Next, the
role of $D$ is played by the second order partial differential operator
$$
D_N:=\frac1{V(x_1,\dots,x_N)}\,\left(D_{x_1}+\dots+
D_{x_N}\right)V(x_1,\dots,x_N)-\const_N,
$$
where $D_{x_i}$ denotes a copy of $D$ acting on variable $x_i$ and
$$
\const_N=m_0+\dots+m_{N-1}.
$$

Although the coefficients of $D_N$ in front of the first order derivatives have
singularities on the diagonals $x_i=x_j$, the operator is well defined on the
space of symmetric polynomials and is diagonalized in the basis $\{p_\nu\}$:
$$
D_Np_\nu=m_\nu p_\nu, \qquad m_\nu:=\sum_{i=1}^N(m_{\nu_i+N-i}-m_{N-i}).
$$

Finally, one can use $D_N$ to define a diffusion process $X_N$ on the space of
$N$-point configurations contained in $\operatorname{supp}W\subseteq R$. Again,
this process has a symmetrizing measure, with density
$$
\prod_{i=1}^N W(x_i)\cdot V^2(x_1,\dots,x_N).
$$

This construction is well known in random matrix literature.  The case of
Hermite polynomials arises from Dyson's Brownian motion model \cite{Dys}. Some
other examples can be found in K\"onig \cite{Koenig}. The construction also
works for some families of discrete orthogonal polynomials, only then $X_N$ is
a jump process.

In the present paper, we make a further step of generalization leading to a
two-parameter family of  infinite-dimensional, continuous time Markov processes
$X^\z$, which are related to the Laguerre polynomials. These words can bring
the reader to believe that the processes $X^\z$ are obtained from the
finite-dimensional Laguerre processes $X_N$ by a large-$N$ limit transition,
but this is not true. Actually, the connection between $X^\z$'s and $X_N$'s is
of a different kind: informally, one can say that the former are related to the
later by analytic continuation in two parameters, dimension $N$ and the
continuous parameter entering the definition of the classical Laguerre
polynomials.

\subsection{The infinite-dimensional Laguerre differential operator and the z-measures}

The operator in question, denoted by $\Dz$, serves as the pre-generator of
process $X^\z$. Initially, $\Dz$ is defined in the algebra of symmetric
functions, $\Sym$, which replaces the algebra of $N$-variate symmetric
polynomials. The elementary symmetric functions $e_1,e_2,\dots$ are
algebraically independent generators of $\Sym$; we use them as independent
variables and define $\Dz:\Sym\to\Sym$ as a second order differential operator
\begin{equation}\label{eq1.A}
\begin{aligned}
\Dz&=\sum_{n\ge1}\left(\sum_{k=0}^{n-1}(2n-1-2k)e_{2n-1-k}e_k\right)
\frac{\pd^2}{\pd e_n^2}\\
&+2\sum_{n'>n\ge1}\left(\sum_{k=0}^{n-1}(n'+n-1-2k)e_{n'+n-1-k}e_k\right)
\frac{\pd^2}{\pd e_{n'}\pd e_n}\\
&+\sum_{n=1}^\infty\big(-ne_n+(z-n+1)(z'-n+1)e_{n-1}\big)\frac{\pd}{\pd e_n}
\end{aligned}
\end{equation}
depending symmetrically on two complex parameters $z$ and $z'$. Recall that the
classical Laguerre polynomials depend on a continuous parameter (the ``Laguerre
parameter'') and so does the $N$-variate Laguerre operator $D_N$. The origin of
operator $\Dz$ is explained in Olshanski \cite{Ols-IMRN12}: it is obtained from
$D_N$ by formal analytic continuation with respect to $N$ and the Laguerre
parameter.

Operator $\Dz$ is diagonalized in a special basis of $\Sym$ formed by the
so-called {\it Laguerre symmetric functions\/}. These functions, denoted by
$\LL^\z_\nu$, depend on parameters $\z$ and are indexed by arbitrary partitions
$\nu=(\nu_1,\nu_2,\dots)$. One has
\begin{equation}\label{eq1.B}
\Dz\LL^\z_\nu=-|\nu|\LL^\z_\nu, \qquad |\nu|:=\nu_1+\nu_2+\dots\,.
\end{equation}
As shown in \cite{Ols-IMRN12}, the Laguerre symmetric functions form an
orthogonal basis in a Hilbert $L^2$ space. Let us explain briefly this point
(for more detail, see \cite{Ols-IMRN12} and Section \ref{sect8.A} below).

So far we treated $\Sym$ as an abstract commutative algebra, freely generated
by elements $e_1,e_2,\dots$, but now we embed it into the algebra of continuous
functions on a topological space, called the {\it Thoma cone\/} and denoted by
$\wt\Om$:
\begin{gather*}
\wt\Om:=\Big\{(\al_1,\al_2,\dots;
\be_1,\be_2,\dots;\de)\in\R^\infty\times\R^\infty\times\R: \\
\al_1\ge\al_2\ge\dots\ge0, \quad \be_1\ge\be_2\ge\dots\ge0, \quad
\sum\al_i+\sum\be_i\le\de\Big\}.
\end{gather*}
Note that the space $\wt\Om$ is locally compact and has infinite dimension in
the sense that its points depend on countably many continuous parameters. The
way of converting elements $F\in\Sym$ into continuous functions $F(\om)$  on
$\wt\Om$ is described in Section \ref{sect7.A}.

Next, we impose the following condition on the parameters:

\begin{condition}\label{cond1.A}
Either both parameters $z$ and $z'$ are complex numbers with nonzero imaginary
part and $z'=\bar z$, or both parameters are real and contained in an open unit
interval of the form $(m,m+1)$ for some $m\in\Z$.
\end{condition}

This is equivalent to requiring that $(z+k)(z'+k)>0$ for every $k\in\Z$. In
particular, Condition \ref{cond1.A} implies that $zz'$ and $z+z'$ are real, so
that the coefficients of operator $\Dz$ are real.

It was shown in \cite{Ols-IMRN12} that for every $\z$ satisfying Condition
\ref{cond1.A}, there exists a unique probability distribution $M^\z$ on
$\wt\Om$ such that all elements of $\Sym$ produce square integrable functions
on $\wt\Om$ with respect to measure $M^\z$, and the Laguerre functions
$\LL^\z_\la$ are pairwise orthogonal with respect to the inner product of the
Hilbert space $L^2(\wt\Om, M^\z)$. In other words, $M^\z$ serves as the {\it
orthogonality measure\/} for the Laguerre symmetric functions. The measures
$M^\z$ appeared even earlier in connection with the problem of harmonic
analysis on the infinite symmetric group; we call them the {\it z-measures\/}
on the Thoma cone.

A difficulty of working with the Laguerre operator $\Dz$ is that its domain as
defined above consists of unbounded functions (more precisely, all the
nonconstant functions from $\Sym$ are unbounded functions on $\wt\Om$). To
overcome this difficulty we modify the domain of definition of the operator in
the following way.

For a triple $\om=(\al,\be,\de)\in\wt\Om$, write $|\om|:=\de$. Let $\F$ stand
for the space of functions on $\wt\Om$ spanned by the functions of the form
$$
e^{-r|\om|}F(\om), \qquad F\in\Sym, \quad r>0.
$$
Such functions are bounded; even more, they vanish at infinity. On the other
hand, $\Dz$ operates on $\F$ in a natural way: here we use the fact that
$|\om|=e_1(\om)$, so that each function from $\F$ is expressed through
variables $e_1,e_2,\dots$.

\subsection{Main results}

Given a locally compact separable metrizable space $E$, denote by $C_0(E)$ the
Banach space of real continuous functions on $E$, vanishing at infinity, with
the supremum norm. A {\it Feller semigroup\/} is a strongly continuous operator
semigroup $T(t)$ on $C_0(E)$ afforded by a transition function $P(t;x,dy)$
(such that $P(t;x,\,\cdot\,)$ is a probability measure),
$$
(T(t)f)(x)=\int_{y\in E}P(t;x,dy)f(y), \qquad x\in E, \quad f\in C_0(E).
$$
A Feller semigroup gives rise to a Markov process on $E$ with c\`adl\`ag
sample trajectories, called a {\it Feller process\/}.

Throughout the paper we assume that $(z,z')$ satisfies Condition \ref{cond1.A}.

\begin{theorem}\label{thm1.A}
{\rm(i)} The differential operator $\Dz$, viewed as an operator on
$C_0(\wt\Om)$ with domain $\F$, is dissipative, and its closure serves as the
generator of a Feller semigroup on $C_0(\wt\Om)$, which we denote by $T^\z(t)$.

{\rm(ii)} The corresponding Feller Markov process $X^\z$ has a unique
stationary distribution, which is the z-measure $M^\z$.
\end{theorem}

Proof is given in Section \ref{sect8}.

Claim (ii) shows that the z-measures $M^\z$ can be characterized as the
stationary distributions of Markov processes $X^\z$.

Taking as the initial distribution for Markov process $X^\z$ its stationary
distribution we get a stationary in time stochastic process, which we denote by
$\wt X^\z$. Theorem \ref{thm1.A} is complemented by the following result,
established in Section \ref{sect9}:

\begin{theorem}\label{thm1.B}
$\wt X^\z$ can be interpreted as a time-dependent determinantal point process
whose correlation kernel can be explicitly computed.
\end{theorem}

Let us explain this claim. Consider the punctured real line
$\R^*:=\R\setminus\{0\}$ and the space $\Conf(\R^*)$ of locally finite point
configurations on $\R^*$. The stationary distribution $M^\z$ can be interpreted
as a probability measure on $\Conf(\R^*)$. More generally, for any finite
collection $t_1<\dots<t_n$ of time moments, the corresponding
finite-dimensional distribution $M^\z(t_1,\dots,t_n)$ of stochastic process
$\wt X^\z$ can be interpreted as a probability measure on the space
$\Conf(\,\underbrace{\R^*\sqcup\dots\sqcup\R^*}_n\,)$. This makes it possible
to describe $M^\z(t_1,\dots,t_n)$ in the language of correlation functions. The
determinantal property claimed in the theorem means that the correlations
functions are given by  $n\times n$ minors extracted from a certain kernel. The
kernel in question, denoted by $K^\z(x,s;y,t)$, has as arguments two space-time
variables, $(x,s)$ and $(y,t)$, where $s\in\R$ and $t\in\R$ are time moments,
while $x\in\R^*$ and $y\in\R^*$ are space positions.

The kernel $K^\z(x,s;y,t)$ appeared first in our paper \cite{BO-PTRF06}, but
there it was derived as the result of a formal limit transition, without
reference to an infinite-dimensional  Markov process. We called $K^\z(x,s;y,t)$
the {\it extended Whittaker kernel\/} to emphasize a similarity with the
well-known dynamical kernels from random matrix theory, the ``extended''
versions of the classical sine, Airy, and Bessel kernels (see Tracy-Widom
\cite{TW-CMP04}).

\subsection{Method of Markov intertwiners}

The results stated above, together with those of \cite{Ols-IMRN12}, were
announced without proofs in the note Olshanski \cite{Ols-POMI10}. The scheme of
the initial proof of Theorem \ref{thm1.A} was the following:

\begin{itemize}
\item Start with the semigroup $\wt T^\z(t)$ in the Hilbert space
$L^2(\wt\Om,M^\z)$ generated by the closure of operator $\Dz$ and show that
$\wt T^\z(t)$ is positivity preserving.

\item Show that $\wt T^\z(t)$ preserves functions from $C_0(\wt\Om)$.

\item Show that the topological support of $M^\z$ is the whole space $\wt\Om$.
\end{itemize}

\noindent The third claim means that the natural map $C_0(\wt\Om)\to
L^2(\wt\Om, M^\z)$ is injective, so that restricting $\wt T^\z(t)$ to
$C_0(\wt\Om)$ gives the desired Feller semigroup $T^\z(t)$.

In the present paper, we use a different approach, based on the {\it method of
Markov intertwiners\/} proposed in Borodin--Olshanski \cite{BO-JFA12}, combined
with the main idea of another recent paper, Borodin--Olshanski \cite{BO-MMJ13}.
To explain this approach, we have first to briefly review what we did in
\cite{BO-JFA12}.

That paper deals with the Gelfand--Tsetlin graph $\GT$ describing the branching
rule for the irreducible characters of unitary groups $U(N)$. The graph is
graded, and its $N$th level $\GT_N$ is a countable set, identified with the
dual object to the unitary group $U(N)$. The graph structure determines a a
sequence of stochastic matrices $\La^2_1,\La^3_2,\dots$, where the $N$th matrix
$\La^{N+1}_N$ has format $\GT_{N+1}\times\GT_N$ and is viewed as a ``link''
connecting the $(N+1)$th and $N$th levels of graph $\GT$. The {\it boundary\/}
of graph $\GT$ is defined as the entrance boundary for the inhomogeneous Markov
chain with varying state spaces $\GT_N$, discrete time parameter ranging over
$\{\dots,3,2,1\}$, and transition function given by the links. The boundary
serves as the space of parameters for the extremal characters of the
infinite-symmetric group $U(\infty)$; this space is a connected,
infinite-dimensional locally compact space. Now, the idea is to find a family
$\{T_N(t): N=1,2,\dots\}$ of Feller semigroups, acting on the spaces
$C_0(\GT_N)$ and compatible with the links in the sense that
$$
T_{N+1}(t)\La^{N+1}_N=\La^{N+1}_NT_N(t), \qquad N=1,2,\dots, \quad t\ge0
$$
(here the operators $T_{N+1}(t)$ and $T_N(t)$ are viewed as matrices of format
$\GT_{N+1}\times\GT_{N+1}$ and $\GT_N\times\GT_N$, respectively). One can say
that the links serve as {\it Markov intertwiners\/} for the semigroups
$T_N(t)$. Given such a family of semigroups, a simple (essentially formal)
argument shows that it gives rise to a ``limit'' Feller semigroup $T_\infty(t)$
generating a Feller process on the boundary.  We showed in \cite{BO-JFA12} that
there is quite a natural way to construct requiring pre-limit semigroups
$T_N(t)$ depending on four additional continuous parameters, and so we obtain a
four-parameter family of limit Feller processes on the boundary.

In the present paper we show that a similar approach works for the Thoma cone
$\wt\Om$. A nontrivial point is what is a suitable substitute of the
Gelfand--Tsetlin graph. As is well known, a natural analog of the
Gelfand--Tsetlin graph is the Young graph, which is the branching graph of the
symmetric group characters. The boundary of the Young graph is an
infinite-dimensional compact space $\Om$, called the {\it Thoma simplex\/}, and
$\wt\Om$ appears as the cone built over $\Om$. Although harmonic analysis on
the infinite symmetric group deals with the Thoma simplex and probability
measures thereof, things go simpler when objects living on $\Om$ are ``lifted''
to $\wt\Om$; this was the main reason for working with the Thoma cone. However,
$\wt\Om$ itself is not a boundary of a branching graph, which was an evident
obstacle for extending the method of \cite{BO-JFA12}.

A solution was found due to the results of \cite{BO-MMJ13}, where we showed
that $\wt\Om$ can be identified with the entrance boundary of a continuous time
Markov chain on the set $\Y$ of all  Young diagrams. This fact enabled us to
apply the formalism of Markov intertwiners with appropriate modifications; in
particular, the discrete index $N=1,2,\dots$ is replaced by continuous index
$r$ ranging over the half-line $\R_{>0}$.

In one direction, the present work goes further than \cite{BO-JFA12}, because
for the processes related to the Gelfand--Tsetlin graph, a result similar to
Theorem \ref{thm1.B} is yet unknown.

\subsection{Comments}

It is natural to compare the results of the present paper to those of
Borodin--Olshanski \cite{BO-PTRF09}, \cite{BO-JFA12}, and Borodin--Gorin
\cite{BG-PTRF12}. In all four papers the authors construct a Feller Markov
process on an infinite-dimensional boundary of a ``projective system''.

The process of \cite{BO-PTRF09} can be obtained by a normalization of the one
we construct here, much similar to the way the Brownian Motion on the sphere
can be obtained from that in the Euclidian space. However, the stationary
distribution of the normalized process does not define a determinantal point
process. Also, in that case the state space is compact, which is much easier to
deal with from the analytic viewpoint.

On the other hand, the process of the present paper is a certain scaling limit
of that from \cite{BO-JFA12}, but in the case of \cite{BO-JFA12} the situation
is more complicated and we were not able to prove there that the time-dependent
correlation functions of the equilibrium process are determinantal (we prove
such a statement in this work). We also do not dispose of an explicit
eigenbasis for the generator there, in contrast to \eqref{eq1.B} above.

The process considered in \cite{BG-PTRF12} was proven to have time-dependent
determinantal structure but it does not possess a stationary distribution,
unlike the three other ones. Also, the underlying state space is quite
different as its coordinates live on a lattice, not on the real line.

Overall, the Markov process we consider in the present paper is the only one so
far that is proven to have all the nice properties one would like to carry over
from the well-known finite dimensional analogs, i.e. Feller property, existence
of a stationary distribution, an explicit description of the (pre)generator and
its eigenbasis, and determinantal formulas for the time-dependent correlations.

To the best of our knowledge, such completeness of the picture was not achieved
in the study of infinite-particle versions of Dyson's Brownian Motion Model
that are also expected to have determinantal time-dependent correlations, see
Jones \cite{LisaJonesThesis08}, Katori--Tanemura \cite{KatoriTanemura-JSP09},
\cite{KatoriTanemura-CMP10}, \cite{KatoriTanemura-MPRF11}, Osada
\cite{Osada-AnnProb13}, \cite{Osada-arXiv12}, Spohn
\cite{Spohn-Dyson-IMAvol87}.

\subsection{Covering  Markov process}
Informally, both the Markov process $X^\z$ on the Thoma cone and its relative,
the Markov process on the boundary of the Gelfand--Tsetlin graph $\GT$, studied
in our paper \cite{BO-JFA12}, may be viewed as interacting particle processes
with {\it nonlocal\/} (or {\it long-range\/}) interaction. On the other hand,
as shown in \cite{BO-JFA12}, the process on the boundary of $\GT$ is
``covered'' by a certain Markov process with local interaction, living on the
path space of $\GT$. In the companion note \cite{BO-Note} we describe a curious
model which conjecturally provides a similar ``covering'' process for $X^\z$.
If the conjectural claims stated in \cite{BO-Note} hold true, this model leads
to an alternative approach to our processes $X^\z$, which looks simple and
intuitively appealing.

\subsection{Organization of the paper}

In Section \ref{sect2} we recall basic facts about Feller semigroups and their
generators, and state a remarkable general theorem from Ethier--Kurtz
\cite{EK05}, which gives a convenient sufficient condition on a matrix of jump
rates ensuring that it generates a Feller Markov chain.

In Section \ref{sect3} we review necessary definitions and facts concerning
convergence of Markov semigroups, taken again from Ethier--Kurtz \cite{EK05}.

Sections \ref{sect4} and \ref{sect5} are devoted to the formalism of Markov
intertwiners (here we present a minimal necessary material and refer to
\cite{BO-JFA12} for more details).

In Section \ref{sect6} we apply the method of Markov intertwiners  to
constructing a concrete one-dimensional diffusion process; our goal here is to
present all the steps of the main construction in a simplified situation.

Short Section \ref{sect7} introduces the Thoma cone and some related objects.

Long Section \ref{sect8} is devoted to the proof of Theorem \ref{thm1.A}; the
argument is developed in strict parallelism with that of Section \ref{sect6}.

Section \ref{sect9} contains the proof of Theorem \ref{thm1.B}.

Finally, in Section \ref{sect10} we briefly describe a Plancherel-type
degeneration of our main construction.

\subsection{Acknowledgements}

A.~B. was partially supported by NSF-grant DMS-1056390. G.~O. was partially
supported by a grant from Simons Foundation (Simons--IUM Fellowship), the
RFBR-CNRS grant 10-01-93114, and the project SFB 701 of Bielefeld University.

\section{Feller semigroups}\label{sect2}

Let $E$ be a locally compact, noncompact, metrizable separable space. Denote by
$C(E)$ the Banach space of real-valued continuous functions on $E$ with the
uniform norm
$$
\Vert f\Vert=\sup_{x\in E}|f(x)|.
$$
Let $C_0(E)\subset C(E)$ denote its closed subspace formed by the functions
vanishing at infinity, and let $C_c(E)$ be the dense subspace of $C_0(E)$
consisting of compactly supported functions.

If $E$ is a discrete countable space, then the continuity requirement
disappears, $C_0(E)$ becomes the space of arbitrary real functions on $E$
vanishing at infinity, and $C_c(E)$ becomes the subspace of finitely supported
functions.

\begin{definition}
A {\it Feller semigroup\/} $\{T(t): t\ge0\}$ is a strongly continuous,
positive, conservative contraction semigroup on $C_0(E)$, see \cite[p.
166]{EK05}.
\end{definition}

Note that in \cite{EK05}, the conservativeness condition is stated in terms of
the semigroup generator. Here are two equivalent reformulations of this
property (see also Liggett \cite[Chapter 3]{Lig}:

\begin{itemize}

\item For any fixed $x\in E$ and $t\ge0$, one has
$$
\sup\{(T(t)f)(x): f\in C_0(E), \; 0\le f\le1\}=1
$$
(if $E$ is compact, then this simply means that $T(t)$ preserves the constant
function $1$).

\smallskip

\item The semigroup admits a transition function, where we mean that a
transition function $P(t\mid x, \,\cdot\,)$ is a probability measure (not a
sub-probability one!) for all $t\ge0$ and $x\in E$.

\end{itemize}

Assume now that $E$ is a countably infinite set and $Q=[Q(a,b)]$ is a matrix of
format $E\times E$ such that
\begin{equation}\label{eq2.8}
\textrm{$Q(a,b)\ge0$ for all $a\ne b$ and $-Q(a,a)=\sum_{b:\, b\ne
a}Q(a,b)<+\infty$ for all $a\in E$}.
\end{equation}
Then there is a constructive way to define a semigroup $\{P_{\min}(t):t\ge0\}$
of {\it substochastic\/} matrices, which provides the {\it minimal solution\/}
to Kolmogorov's backward and forward equations,
$$
\frac{d}{dt}\, P(t)=Q P(t),\qquad \frac{d}{dt}\, P(t)=P(t)Q,
$$
see Feller \cite{Fel40} and Liggett \cite[Chapter 2]{Lig}.

\begin{definition}
One says that $Q$ is {\it regular\/} if the matrices $P_{\min}(t)$ from the
minimal solution are stochastic.
\end{definition}

If the $Q$-matrix is regular, then $P_{\min}(t)$ is a unique solution to both
the backward and forward Kolmogorov equations. Qualitatively, regularity of the
$Q$-matrix means that the Markov chain is non-exploding: one cannot escape to
infinity in finite time.

\smallskip

Recall a few general notions (see Ethier--Kurtz \cite[Chapter 1, Sections
1--3]{EK05}). Any strongly continuous contractive semigroup on a Banach space
is uniquely determined by its {\it generator\/}, which is a  densely defined
closed dissipative operator. We will denote generators by symbol $A$ (possibly
with additional indices), and $\dom A$ will denote the domain of $A$. A {\it
core\/} of a generator $A$ is a subspace $\F\subseteq\dom A$ such that the
closure of the operator $A|_{\F}$ (the restriction of $A$ to $\F$) coincides
with $A$ itself; thus $A$ is uniquely determined by its restriction to a core.
It often happens that an explicit description of $\dom(A)$ is unavailable but
one can write down the action of $A$ on a core $\F$, and then the {\it
pre-generator\/} $A|_{\F}$ serves as a substitute of $A$.

We will need a result from Ethier--Kurtz \cite{EK05} which provides a
convenient sufficient condition of regularity together with important
additional information:

\begin{theorem}\label{thm2.1}
Let $E$ be a countably infinite set and $Q=[Q(a,b)]$ be a matrix of format
$E\times E$ satisfying \eqref{eq2.8}.  Assume additionally that $Q$ has
finitely many nonzero entries in every row and every column, and there exist
strictly positive functions $\ga(a)$ and $\eta(a)$ on $E$ that tend to
$+\infty$ at infinity and are such that
\begin{gather}
-Q(a,a)\le C\ga(a), \quad \forall a\in E, \label{eq2.3}\\
 Q\frac1{\ga}\le \frac C\ga \quad \text{\rm pointwise}\label{eq2.1}\\
Q\eta\le C\eta \quad \text{\rm pointwise} \label{eq2.2}
\end{gather}
where $C$ is a positive constant and, for an arbitrary function $f(a)$ on $E$,
the notation $Qf$ means the function
$$
(Qf)(a)=\sum_{b\in E}Q(a,b)f(b)=\sum_{b\in E, \, b\ne a}Q(a,b)(f(b)-f(a)),
$$
the sum being finite because of the row finiteness condition.

Under these hypotheses we have{\rm:}

{\rm(i)} $Q$ is regular and so determines a Markov semigroup $P(t)$.

{\rm(ii)} This semigroup induces a Feller semigroup $\{T(t):t\ge0\}$ on
$C_0(E)$.

{\rm(iii)} Let $A$ denote the generator of $T(t)$; its domain  $\dom(A)$
consists of those functions $f\in C_0(E)$ for which $Qf\in C_0(E)$. Moreover,
$A=Q$ on $\dom A$.

{\rm(iv)} The subspace $C_c(E)\subset C_0(E)$ of compactly supported functions
is a core for $A$.
\end{theorem}

\begin{proof}
This is an adaptation of Theorem 3.1 in \cite[Chapter 8]{EK05}, which actually
holds under less restrictive assumptions.
\end{proof}

\section{Convergence of semigroups and Markov processes}\label{sect3}

\subsection{Convergence of semigroups}
Let $I$ be one of the sets $\R_{>0}$ (strictly positive real numbers) or
$\Z_{>0}$ (strictly positive integers). Assume that $\{\LLL_r:r\in I\}$ is a
family of real Banach spaces, $\LLL_\infty$ is one more real Banach space, and
for every $r\in I$ we are given  a contractive linear operator
$\pi_r:\LLL_\infty\to \LLL_r$. If $f$ is a vector of one of these spaces, then
$\Vert f\Vert$ denotes its norm.

\begin{definition}\label{def3.A}
We say that {\it vectors $f_r\in \LLL_r$ approximate a vector
$f\in\LLL_\infty$\/} and write $f_r\to f$ if
$$
\lim_{r\to\infty}\Vert f_r-\pi_r f\Vert=0.
$$
\end{definition}

\begin{definition}\label{def3.B}
Let $\{T_\infty(t): t\ge0\}$ and $\{T_r(t):t\ge0\}$ be strongly continuous
contraction semigroups on $\LLL_\infty$ and $\LLL_r$. We say that {\it the
semigroups $T_r(t)$ approximate the semigroup $T_\infty(t)$\/} and write
$T_r(t)\to T_\infty(t)$ if
\begin{equation}\label{eq3.1}
\lim_{r\to\infty}\sup_{0\le t\le t_0}\Vert T_r(t)\pi_r f-\pi_r
T_\infty(t)f\Vert= 0 \qquad \textrm{for all $f\in\LLL_\infty$ and any $t_0>0$}.
\end{equation}
\end{definition}

Our aim is to check this condition using an appropriate convergence of
semigroup generators. So let $A_\infty$ and $A_r$ denote the generators of the
above semigroups and let  $\dom(A_\infty)$, $\dom(A_r)$ be the domains of the
generators.

\begin{definition}\label{def3.C}
Fix a core $\F\subseteq\dom(A)$. We say that the operator $A_\infty\vert_\F$ is
{\it approximated\/} by the operators $A_r$ if for any vector $f\in\F$  one can
find a family of vectors $\{f_r\in \dom(A_r): r\in I\}$ such that $f_r\to f$,
and $A_r f_r\to A_\infty f$ as $r\to\infty$.
\end{definition}

In other words, this kind of operator convergence means that every vector from
the graph of $A_\infty\vert_\F$ can be approximated by vectors from the graphs
of the operators $A_r$.

\begin{theorem}\label{thm3.2}
Let $T_\infty(t)$, $T_r(t)$, $A_\infty$, $A_r$, and $\F$ be as above. If
$A_\infty\vert_\F$ is approximated by the operators $A_r$, then $T_r(t)\to
T_\infty(t)$ in the sense of Definition \ref{def3.B}.
\end{theorem}

\begin{proof}
For $I=\Z_{>0}$, this is part of Ethier--Kurtz \cite[Chapter 1, Theorem
6.1]{EK05}. The case $I=\R_{>0}$ is immediately reduced to the case
$I=\Z_{>0}$, because condition \eqref{eq3.1} is equivalent to saying that the
same limit relation holds along any sequence of positive real numbers tending
to $+\infty$.
\end{proof}

\subsection{Convergence of Markov processes}
Below we use the term {\it Markov process\/} as a shorthand for a Markov family
which may start from any given point of the state space or from any given
initial probability distribution. We are dealing exclusively with processes
stationary in time and with infinite life time.

Given an initial distribution $M(0)$ of a Markov process on a space $E$, one
may speak about its finite-dimensional distributions $M(t_1,\dots,t_k)$
corresponding to any prescribed time moments $0\le t_1<\dots<t_k$,
$k=1,2,\dots$\,. Every such distribution $M(t_1,\dots,t_k)$ is a probability
measure on the $k$-fold direct product $E^k=E\times\dots\times E$.

Let $E$ be a locally compact metrizable space and $T(t)$ be a Feller semigroup
on $C_0(E)$; then $T(t)$ gives rise to a Markov process $X(t)$ on $E$ with
c\`adl\`ag sample trajectories, see Ethier--Kurtz \cite[Chapter 4, Section
2]{EK05}. The finite-dimensional distributions of $X(t)$ are determined by the
semigroup $T(t)$ in the following way: For arbitrary functions
$g_1,\dots,g_k\in C_0(E)$, define recursively functions $h_k,\dots,h_0$ by
\begin{multline}\label{eq3.5}
h_k=g_k, \quad h_{k-1}=g_{k-1}\cdot(T(t_k-t_{k-1})h_k),\, \dots\\
\dots,\, h_1=g_1\cdot(T(t_2-t_1)h_2), \quad h_0=T(t_1)h_1,
\end{multline}
where dots mean pointwise product, so that $h_{k-1}$ is obtained by applying
operator $T(t_k-t_{k-1})$ to $h_{k-1}$ and then multiplying the resulting
function by $g_{k-1}$, etc. Then
\begin{equation}\label{eq3.6}
\langle g_1\otimes\dots\otimes g_k, M(t_1,\dots,t_k)\rangle= \langle h_0,
M(0)\rangle,
\end{equation}
where the angle brackets denote the canonical pairing between functions and
measures, and $(g_1\otimes\dots\otimes g_k)(x_1,\dots,x_k)=g_1(x_1)\dots
g_k(x_k)$ for $(x_1,\dots,x_k)\in E^k$ (this is a function from $C_0(E^k)$).

Let $X_r(t)$ and $X(t)$ be Markov processes with state spaces $E_r$ and $E$,
respectively (as before,  $r$ ranges over the index set $I$, which is either
$\R_{>0}$ or $\Z_{>0}$). Assume that $E$ is a locally compact metrizable
separable space and each $E_r$ is realized as a discrete locally finite subset
of $E$. Further, assume that as $r\to\infty$, $E_r$ becomes more and more dense
in $E$; more precisely, we postulate that any probability measure $P$ on $E$
can be represented as the weak limit $w$-$\lim_{r\to\infty}P_r$, where $P_r$ is
a probability measure supported by $E_r$.

\begin{definition}\label{def3.D}
Under these assumptions we say that the processes $X_r(t)$ {\it approximate\/}
the process $X(t)$ and write  $X_r(t)\to X(t)$ if whenever an initial
distribution $M(0)$ for the process $X(t)$ is represented as a weak limit of a
family $\{M_r(0)\}$ of initial distributions of processes $X_r(t)$, we have
$$
w\text{-}\lim_{r\to\infty}M_r(t_1,\dots,t_k)=M(t_1,\dots,t_k),
$$
meaning weak convergence on $E^k$ of the finite-dimensional distributions
corresponding to any given time moments $0<t_1<\dots<t_k$, $k=1,2,\dots$\,.
\end{definition}

\begin{corollary}\label{cor3.1}
Under the above assumptions, assume additionally that the Markov processes
$X_r(t)$ and $X(t)$ come from some Feller semigroups on the Banach spaces
$\LLL_r=C_0(E_r)$ and $\LLL=C_0(E)$, respectively. Further, let the projection
$\pi_r:\LLL\to\LLL_r$ be defined as the restriction map from $E$ to $E_r$.

If the hypotheses of Theorem \ref{thm3.2} are satisfied, then $X_r(t)\to X(t)$
in the sense of Definition \ref{def3.D}.
\end{corollary}

Note that $\pi_r$ is well defined as a map from $C_0(E)$ to $C_0(E_r)$ because
$E_r$ is assumed to be locally finite, so that if a sequence of points goes to
infinity along $E_r$ then it also goes to infinity in $E$.

\begin{proof}
It suffices to prove that
\begin{equation}\label{eq3.7}
\lim_{r\to\infty}\langle g_1\otimes\dots\otimes g_k,
M_r(t_1,\dots,t_k)\rangle=\langle g_1\otimes\dots\otimes g_k,
M(t_1,\dots,t_k)\rangle
\end{equation}
for any collection $g_1,\dots,g_k\in C_0(E)$, because the functions of the form
$g_1\otimes\dots\otimes g_k$ are dense in $C_0(E^k)$.

Let $h_k,\dots,h_0\in C_0(E)$ be defined as in \eqref{eq3.5} and, for each
$r\in I$, let $h_{k;r},\dots,h_{0;r}\in C_0(E_r)$ be defined in the same way,
starting from the collection
$$
g_{1;r}:=\pi_r(g_1),\; g_{2;r}:=\pi_r(g_2),\; \dots,\; g_{k;r}:=\pi_r(g_k).
$$
By virtue of \eqref{eq3.6}, the desired limit relation \eqref{eq3.7} is
equivalent to
\begin{equation*}
\lim_{r\to\infty}\langle h_{0;r}, M_r(0)\rangle=\langle h_0, M(0)\rangle
\end{equation*}
Since $w$-$\lim_{r\to\infty}M_r(0)=M(0)$ by assumption, it suffices to prove
that
$$
\lim_{r\to\infty}\Vert h_{0;r}-\pi_r h_0\Vert= 0.
$$
To do this, we prove step by step that
$$
\lim_{r\to\infty}\Vert h_{i;r}-\pi_r h_i\Vert= 0,
$$
for $i=k,\dots,0$, where each transition $i\to i-1$ is justified by making use
of Theorem \ref{thm3.2}.
\end{proof}

This argument is patterned from the proof of Theorem 2.5 in \cite[Chapter
4]{EK05}. Note also that another kind of convergence is established in
\cite[Chapter 4, Theorem 2.11]{EK05}.

\section{Feller projective systems}\label{sect4}

\subsection{Links}

Let $E'$ and $E$ be two measurable spaces. Recall that a  {\it Markov kernel\/}
linking $E'$ to $E$ is a function $\La(\,\cdot\,,\,\cdot\,)$ in two variables,
one ranging over $E'$ and the other ranging over measurable subsets of $E$,
such that $\La$ is measurable with respect to the first argument and is a
probability measure relative to the second argument. We use the notation $\La:
E'\dasharrow E$ and call $\La$ a {\it link\/} between $E'$ and $E$.

If $E$ is a discrete set, then, setting $\La(x,y):=\La(x,\{y\})$, we may regard
$\La$ as a function on $E'\times E$. If both $E'$ and $E$ are discrete, then
$\La$ is simply a stochastic matrix of format $E'\times E$.

The operation of composition of two links $E''\dasharrow E'$ and $E'\dasharrow
E$ is defined in a natural way: denoting the first link by $\La^{E''}_{E'}$ and
the second one by $\La^{E'}_E$ we have
$$
(\La^{E''}_{E'}\La^{E'}_E)(x,dz)=\int_{y\in
E'}\La^{E''}_{E'}(x,dy)\La^{E'}_E(y,dz).
$$
In the discrete case this operation reduces to conventional matrix product.

The possibility of composing links makes it possible to regard them as
morphisms in a category whose objects are measurable spaces, see
\cite{BO-MMJ13}. However, links are not ordinary maps; this is why we denote
them by the dash arrow.

A link $\La:E'\dasharrow E$ takes a probability measure $M$ on $E'$ to a
probability measure $M\La$ on $E$:
$$
(M\La)(dy)=\int_{x\in E'}M'(dx)\La(x,dy).
$$
If both spaces are discrete then measures may be viewed as row-vectors and then
the product $M\La$ becomes the conventional product of a row-vector by a
matrix.

Dually, $\La$ determines a contractive linear map $B(E)\to B(E')$ between the
Banach spaces of bounded measurable functions, denoted as $F\mapsto \La F$:
$$
(\La F)(x)=\int_{y\in E}\La(x,dy)F(y).
$$
In the discrete case, functions may be viewed as column-vectors and then $\La
F$ becomes the conventional product of a matrix by a column-vector.

We say that a link $\La: E'\to E$  between two locally compact spaces is a {\it
Feller link\/} if the corresponding linear map $B(E)\to B(E')$ sends
$C_0(E)\subset B(E)$ to $C_0(E')\subset B(E')$.

If $E$ is discrete, then this condition means that for any fixed $y\in E$, the
function $x\mapsto\La(x,y):=\La(x,\{y\})$ on $E'$ lies in $C_0(E')$.

\subsection{Projective systems and boundaries}

Let, as above, $I$ denote one of the two sets $\R_{>0}$  or $\Z_{>0}$. By a
{\it projective system\/} with index set $I$ we mean a family $\{E_r: r\in I\}$
of discrete spaces together with a family of links
$\{\La^{r'}_r:E_{r'}\dasharrow E_r: r'>r\}$, where every $E_r$ is finite or
countably infinite, and for any triple $r''>r'>r$ of indices one has
$\La^{r''}_{r'}\La^{r'}_r=\La^{r''}_r$; see \cite{BO-MMJ13}. If $I=\Z_{>0}$,
then it suffices to specify the links $\La^{r'}_r$ for neighboring indices
$r'=r+1$ and then set
$$
\La^{r'}_r:=\La^{r'}_{r'-1}\dots\La^{r+1}_r
$$
for arbitrary couples $r'>r$.

(The above definition is applicable to more general ordered index sets but we
would like to avoid excessive formalism. For the purpose of the present paper
we need the continuous index set $I=\R_+$. Concrete projective systems with
discrete index sets are considered in \cite{BO-JFA12} and \cite{BO-MMJ13}. In
some general considerations (see below) the case $I=\R_{>0}$ is readily reduced
to that of $I=\Z_{>0}$.)

Following \cite{BO-MMJ13}, we define the {\it boundary\/} $E_\infty$ of a
projective system $\{E_r, \La^{r'}_r\}$ in the following way. Consider the
projective limit space $\varprojlim\M(E_r)$, where $\M(E_r)$ stands for the set
of probability measures on $E_r$ and the limit is taken with respect to the
projections $\M(E_{r'})\to\M(E_r)$ induced by the links $\La^{r'}_r$. Assuming
that the projective limit space is nonempty, we take as $E_\infty$ the set of
its extreme points.

We refer to \cite{BO-MMJ13} for more details.  Note that $\M(E_r)$ may be
viewed as a simplex with vertex set $E_r$, and every projection
$\M(E_{r'})\to\M(E_r)$ is an affine map of simplices (that is, it preserves
barycenters), so our projective limit space is a projective limit of simplices.

By the very definition of projective limit, an element of $\varprojlim\M(E_r)$
is a family $\{M_r\in\M(E_r): r\in I\}$ of probability measures satisfying the
relation $M_{r'}\La^{r'}_r=M_r$ for every couple of indices $r'>r$. Such a
family is called a {\it coherent system\/} of measures.

As explained in \cite{BO-MMJ13}, there is a canonical bijection
\begin{equation}\label{eq4.A}
\M(E_\infty)\,\longleftrightarrow\,\varprojlim\M(E_r),
\end{equation}
where $\M(E_\infty)$ denotes the space of probability measures on $E_\infty$.
This means that for every $r\in I$ there is a link $\La^\infty_r:E_\infty\to
E_r$ such that the correspondence $ M_\infty\mapsto \{M_r:r\in I\}$ given by
$M_r:=M_\infty\La^\infty_r$ establishes a one-to-one correspondence between
probability measures on the boundary and coherent families of probability
measures. We say that $M_\infty$ is the {\it boundary measure\/} for the
coherent system $\{M_r\}$.

Obviously, the links $\La^\infty_r$ are compatible with the links $\La^{r'}_r$
in the sense that
$$
\La^\infty_{r'}\La^{r'}_r=\La^\infty_r \quad \textrm{for any $r'>r$}.
$$

Observe that in the case of $I=\R_{>0}$ the boundary does not change if in the
above construction we will assume that the indices range along an arbitrary
fixed sequence of strictly increasing real numbers converging to $+\infty$.
This enables one to reduce the case $I=\R_{>0}$ to that of $I=\Z_{>0}$. For
further reference, let us call this simple trick {\it discretization of the
index set\/}.

\subsection{Running example: The binomial projective system $\mathbb
B$}\label{sect4.A}

In this illustrative example taken from Borodin--Olshanski \cite{BO-MMJ13}, the
index set $I$ is $\R_{>0}$; for every index $r\in\R_{>0}$ the corresponding
discrete set $E_r$ is a copy of $\Z_+:=\{0,1,2,\dots\}$; and for every two
indices $r'>r$ the corresponding link $\Z_+\dasharrow\Z_+$ is given by
$$
\LB^{r'}_r(l,m)=\frac{l!}{m!(l-m)!}\left(\frac r{r'}\right)^m\left(1-\frac
r{r'}\right)^{l-m}, \qquad l,\, m\in\Z_+.
$$
Note that $\La^{r'}_r(l,\,\cdot\,)$ is a binomial distribution on the set $\{m:
0\le m\le l\}$. For this reason we call this system the  {\it binomial
projective system.}

As shown in \cite{BO-MMJ13}, its boundary $E_\infty$ can be identified with the
halfline $\R_+$ (the set of nonnegative real numbers) and the links
$\La^\infty_r:\R_+\to\Z_+$ are given by Poisson distributions:
$$
\LB^\infty_r(x,m)=e^{-rx}\frac{(rx)^m}{m!}, \qquad x\in\R_+, \quad m\in\Z_+.
$$

\subsection{Feller projective systems}\label{sect4.B}

Let $\{E_r, \La^{r'}_r\}$ be a projective system as defined above. Equip the
boundary $E_\infty$ with the {\it intrinsic topology\/} --- the weakest one in
which all functions of the form
$$
x\mapsto \La^\infty_r(x,y), \qquad r\in I, \quad y\in E_r,
$$
are continuous. We say that $\{E_r, \La^{r'}_r\}$  is a {\it Feller system\/}
if the following three conditions are satisfied:

(1) All links $\La^{r'}_r$ are Feller.

(2) The boundary  $E_\infty$ is a locally compact Hausdorff space with respect
to the intrinsic topology.

(3) In this topology, all links $\La^\infty_r$ are Feller.

\medskip

Note that under condition (1), the definition of the intrinsic topology is not
affected by discretization of the index set, which entails that the intrinsic
topology is automatically metrizable with countable base.

\medskip

As an illustration, let us check that the binomial projective system from our
running example (see Section \ref{sect4.A} above) is a Feller system.

Indeed, from the very definition of the ``binomial'' links $\La^{r'}_{r}$ and
``Poissonian'' links $\La^\infty_r$ it is clear that they are Feller links. It
remains to check that the intrinsic boundary topology on $\R_+$ is the
conventional topology and so is locally compact.

By the very definition, the intrinsic topology is the weakest one in which all
the functions $x\mapsto \La^\infty_r(x,m)$, where parameter $r$ ranges over
$\R_+$ and parameter $m$ ranges over $\Z_+$, are continuous. We will prove a
stronger claim: even if only $m$ varies but $r>0$ is chosen arbitrarily and
fixed, then the corresponding topology coincides with the conventional one.

To do this, consider the map $\R_+\to [0,1]^\infty$ assigning to $x\in\R_+$ the
sequence
$$
\{a_m(x): m\in\Z_+\}, \qquad
a_m(x):=\La^\infty_r(x,m)=e^{-rx}\frac{(rx)^m}{m!}.
$$
This map is injective, for $x$ is recovered from $\{a_m(x)\}$ from the identity
$$
\sum_{m=0}^\infty s^m a_m(x)=e^{(s-1)rx}.
$$
By the very definition, the weakest topology on $\R_+$ making all the functions
$a_m(x)$ continuous is exactly the topology induced by the embedding of $\R_+$
into the cube $[0,1]^\infty$ equipped with the product topology.

Observe now that the cube $[0,1]^\infty$ is compact and the above map extends
by continuity to the one-point compactification $\R_+\cup\{+\infty\}$ of $\R_+$
by setting  $a_m(+\infty)=0$ for all $m$. Obviously, the extended map is
injective, too. Therefore, it is a homeomorphism onto a closed subset of
$[0,1]^\infty$. This implies the desired claim.

\subsection{The density lemma}\label{sect4.C}
If $\{E_r,\La^{r'}_r\}$ is a Feller projective system with boundary $E_\infty$,
then the subspace
$$
\bigcup_{r\in I}\La^\infty_r C_0(E_r)\subset C_0(E_\infty)
$$
is dense in the norm topology; see Borodin--Olshanski \cite[Lemma
2.3]{BO-JFA12}. Here $\La^\infty_r C_0(E_r)$ denotes the range of the operator
$\La^\infty_r: C_0(E_r)\to C_0(E_\infty)$.

For further reference we call this assertion the {\it density lemma\/}. Its
proof is simple; it relies on the fact that for a locally compact space $E$,
the vector space of (signed) measures on $E$ with finite total variation  is
the Banach dual to $C_0(E)$.

Since $C_c(E_r)$ is dense in $C_0(E_r)$ and the operator $\La^\infty_r:
C_0(E_r)\to C_0(E_\infty)$ is contractive, the density lemma is equivalent to
the assertion that the set of functions of the form
$$
x\mapsto \La^\infty_r(x,y), \qquad r\in I, \quad y\in E_r,
$$
is {\it total\/} in $C_0(E_\infty)$ meaning that the linear span of these
functions is dense.

\medskip

For our running example, the latter assertion means that the set of functions
$$
e^{-rx} x^n, \qquad r>0, \quad n\in\Z_+
$$
is total in $C_0(\R_+)$. But here a stronger claim holds: it is not necessary
to take all $r>0$, we may assume that $r$ is {\it fixed\/}. In other words, for
any fixed $r>0$, the space of polynomials in $x$ multiplied by the exponential
$e^{-rx}$ is dense in $C_0(\R_+)$; see \cite[Corollary 3.1.6]{BO-MMJ13} for a
simple proof. Thus, in this situation, $\La^\infty_r C_0(E_r)\subset
C_0(E_\infty)$ is dense for any fixed $r$. However, this is a special property
of the projective system under consideration; for instance, it does not hold in
the context of \cite{BO-JFA12}.

\subsection{Approximation of boundary measures}

Our definition of the boundary measure $M_\infty$ as a limit of a coherent
system of measures $M_r$ was purely formal. Here we show that, under a suitable
additional assumption, $M_\infty$ is a limit of $\{M_r\}$ in a conventional
sense.

Let, as above, $\{E_r,\La^{r'}_r\}$ be a Feller projective system with boundary
$E_\infty$, and adopt the following assumption:

\begin{condition}\label{cond4.A}
For every $r\in I$ there exists an embedding $\varphi_r:E_r\hookrightarrow
E_\infty$ such that:

(i) The image $\varphi_r(E_r)$ is a discrete subset in $E_\infty$.

(ii) For any fixed $s\in I$ and any fixed $y\in E_s$
$$
\lim_{r\to\infty}\sup_{x\in
E_r}\left|\La^r_s(x,y)-\La^\infty_s(\varphi_r(x),y)\right|=0.
$$
\end{condition}

So far our measures lived on varying spaces. Now, using the maps $\varphi_r$,
we can put all them on one and the same space, the boundary $E_\infty$. Namely,
we simply replace $M_r$ with its pushforward $\varphi_r(M_r)$, which is a
probability measure on $E_\infty$. A natural question is whether the resulting
measures converge to $M_\infty$, and the next proposition gives an affirmative
answer.

\begin{proposition}\label{prop4.A}
Assume that Condition \ref{cond4.A} is satisfied. Let $\{M_r:r\in I\}$ be a
coherent system of probability distributions and $M_\infty$ be the
corresponding boundary measure. As $r\to\infty$, the measures $\varphi_r(M_r)$
converge to $M_\infty$ in the weak topology.
\end{proposition}

Note that for this proposition, part (i) of the condition is not relevant, but
it will be used in the sequel (see Section \ref{sect5.A}).

\begin{proof}
We have to show that for any bounded continuous function $F$
$$
\langle F,\varphi_r(M_r)\rangle \to \langle F, M_\infty\rangle.
$$
Since all the measures in question are probability measures, we may replace the
weak convergence by the vague convergence, that is, we may assume that $F$ lies
in the space $C_0(E_\infty)$. Next, we apply the density lemma (see Section
\ref{sect4.C}), which enables us to further assume that $F$ has the form
$F(x)=\La^\infty_s(x,y)$ for some fixed $s\in I$ and $y\in E_s$. Then we get
$$
\langle F, M_\infty\rangle=\int_{x\in
E_\infty}M_\infty(dx)\La^\infty_s(x,y)=M_s(y).
$$

On the other hand,
\begin{equation}\label{eq4.B}
\langle F, \varphi_r(M_r)\rangle=\langle F\circ \varphi_r, M_r\rangle.
\end{equation}
Here the function $F\circ \varphi_r$ lives on $E_r$, and for $x\in E_r$ one can
write
$$
(F\circ\varphi_r)(x)=F(\varphi_r(x))=\La^\infty_r(\varphi_r(x),y)=\La^r_s(x,y)+\epsi(r,x),
$$
where, by virtue of Condition \ref{cond4.A}, the remainder term $\epsi(r,x)$
tends to 0 uniformly on $x$, as $r\to\infty$. Therefore, \eqref{eq4.B} equals
$$
\langle \La^r_s(\,\cdot\,, y), M_r\rangle +\ldots=M_s(y)+\ldots,
$$
where the dots denote a remainder term converging to 0. This completes the
proof.
\end{proof}

\begin{example}\label{ex4.A}
Consider the projective system $\mathbb B$ introduced in Section \ref{sect4.A}.
Recall that then the index set $I$ is $\R_{>0}$, $E_r=\Z_+$ for all $r>0$, and
the boundary $E_+$ is $\R_+$. Define the map $\varphi_r:E_r\to E_\infty$ as
$$
\varphi_r(l)=r^{-1}l, \qquad l\in\Z_+,
$$
and let us check that Condition \ref{cond4.A} is satisfied.

Indeed, in our situation it means that that for fixed $s>0$ and $m\in\Z_+$
\begin{equation}\label{eq4.C}
\lim_{r\to \infty}\sup_{l\in
\Z_+}|\LB^{r}_{s}(l,m)-\LB^\infty_{s}(r^{-1}l,m)|=0.
\end{equation}
The explicit expressions for the links in question are (see Section
\ref{sect4.A}):
\begin{gather*}
\LB^{r}_s(l,m)=\frac{l!}{m!(l-m)!}\left(\frac sr\right)^m\left(1-\frac
sr\right)^{l-m}, \qquad l,\, m\in\Z_+,\\
\LB^\infty_s(x,m)=e^{-sx}\frac{(sx)^m}{m!}, \qquad x\in\R_+, \quad m\in\Z_+.
\end{gather*}
In \eqref{eq4.C}, set $x=r^{-1}l$ and note that
$$
\frac{l!}{(l-m)!r^m}=x^m\left(1+O(r^{-1})\right), \quad \left(1-\frac
sr\right)^{l-m}=\left(1-\frac sr\right)^{rx}\left(1+O(r^{-1})\right).
$$
Therefore, \eqref{eq4.C} follows from the fact that (see \cite[Lemma
3.1.4]{BO-MMJ13})
$$
\textrm{$\lim_{r\to+\infty}\left(1-\frac sr\right)^{rx}x^m = e^{-rx}x^m$
uniformly on $x\in\R_+$}.
$$

For this example, Proposition \ref{prop4.A} gives a specific recipe for
approximating arbitrary probability measures on $\R_+$ by atomic measures
supported by the grids $r^{-1}\Z_+$.

\end{example}

\section{Boundary Feller semigroups: general formalism}\label{sect5}

In this section,  $\{E_r, \La^{r'}_r\}$ is a Feller projective system with
index set $I$ equal to $\R_{>0}$ or $\Z_{>0}$, and boundary $E_\infty$.

\subsection{Intertwining of semigroups}\label{sect5.B}

Let $E'$ and $E$ be two locally compact metrizable spaces, $T'(t)$ and $T(t)$
be Feller semigroups on $C_0(E')$ and $C_0(E)$, respectively, and $\La:
E'\dasharrow E$ be a Feller link. Let us say that $\La$ {\it intertwines\/} the
semigroups $T'(t)$ and $T(t)$ if
\begin{equation}\label{eq5.A}
T'(t)\Lambda=\Lambda T(t), \quad  t\ge0,
\end{equation}
where both sides are interpreted as operators $C_0(E)\to C_0(E')$.

\begin{proposition}\label{prop5.B}
Assume that for every $r\in I$ we are given a Feller semigroup $\{T_r(t):
t\ge0\}$ on $C_0(E_r)$. Assume further that the links $\La^{r'}_r$ intertwine
the corresponding semigroups, that is, for any two indices $r'>r$
\begin{equation}\label{eq5.4}
T_{r'}(t)\Lambda^{r'}_{r}=\Lambda^{r'}_{r}T_{r}(t).
\end{equation}
Then the there exists a unique Feller semigroup $\{T_\infty(t): t\ge0\}$ on
$E_\infty$ such that $\La^\infty_r$ intertwines $T_\infty(t)$ and $T_r(t)$ for
every $r\in I$,
\begin{equation}\label{eq5.5}
T_\infty(t)\Lambda_r^\infty=\Lambda^\infty_r T_r(t),\qquad t\ge 0.
\end{equation}
\end{proposition}

\begin{proof}
In the case $I=\Z_{>0}$ this assertion was established in \cite[Proposition
2.4]{BO-JFA12}. The same argument works in the case $I=\R_{>0}$.
\end{proof}

We call the semigroup $T_\infty(t)$ constructed in the above proposition the
{\it boundary semigroup\/}. Now we are going to describe its generator.

We start with the simple observation that relation \eqref{eq5.A} has an
infinitesimal analog: namely, denoting by  $A'$ and $A$ the generators of the
semigroups $T'(t)$ and $T(t)$ from \eqref{eq5.A}, one has
$$
\La:\dom(A)\to \dom(A')
$$
and
\begin{equation}\label{eq5.B}
A'\La=\La A.
\end{equation}
In words, if a Feller link intertwines two Feller semigroups, then it also
intertwines their generators. Indeed, this is an immediate consequence of the
very definition of the semigroup generator.

\begin{proposition}\label{prop5.A}
Let the semigroups $T_r(t)$ be as in the above proposition, $T_\infty(t)$ be
the corresponding boundary semigroup, and $A_r$ and $A_\infty$ denote the
generators of these semigroups. Take for each $r\in I$ an arbitrary core
$\F_r\subseteq \dom(A_r)$ for the operator $A_r$; then the linear span of the
vectors of the form $\La^\infty_r f$, where $r$ ranges over $I$ and $f$ ranges
over $\F_r$, is a core for $A_\infty$.
\end{proposition}

Note that the action of $A_\infty$ on such a core is determined according to
\eqref{eq5.B}, that is
\begin{equation}\label{eq5.C}
A_\infty\La^\infty_r f=\La^\infty_r A_r f, \qquad f\in\dom(A_r).
\end{equation}

\begin{proof}
We will apply a well-known characterization of cores based on Hille--Yosida's
theorem: Let $A$ be the generator of a strongly continuous contraction
semigroup on a Banach space; a subspace $\F\subseteq\dom(A)$ is a core for $A$
if and only if, for any constant $c>0$, the subspace $(c-A)\F$ is dense. The
proof is simple (cf. \cite[Chapter 1, Proposition 3.1]{EK05}). Indeed, fix an
arbitrary $c>0$. By Hille--Yosida's theorem, the operator $(c-A)^{-1}$ is
defined on the whole space and bounded. Next, the closure of $A|_{\F}$
coincides with $A$ if and only if the closure of $(c-A|_{\F})^{-1}$ coincides
with $(c-A)^{-1}$, and this in turn just means that $(c-A)\F$, which is the
domain of $(c-A|_{\F})^{-1}$, is dense.

Take now as $\F$ the linear span of the union of the subspaces
$\La^\infty_r\F_r$. We already know that $\F$ is contained in $\dom(A_\infty)$.

By the criterion above, it suffices to prove that $(c-A_\infty)\F$ is dense in
$C_0(E_\infty)$ for any $c>0$. We have
$$
(c-A_\infty)\F=\operatorname{span}\left(\bigcup_{r\in
I}(c-A_\infty)\La^\infty_r\F_r\right) =\operatorname{span}\left(\bigcup_{r\in
I}\La^\infty_r(c-A_r)\F_r\right),
$$
where the last equality follows from \eqref{eq5.C}. On the other hand, we know
that for every $r\in I$, $(c-A_r)\F_r$ is dense in $C_0(E_r)$, because $\F_r$
is a core for $A_r$. Therefore, the closure of $(c-A_\infty)\F$ coincides with
the closure of the subspace $\bigcup_{r\in I} \La^\infty_r C_0(E_r)$. But the
latter subspace is dense by Proposition \ref{prop5.B}. Therefore,
$(c-A_\infty)\F$ is dense, too.
\end{proof}

Let us return to the basic intertwining relation \eqref{eq5.A}. Under suitable
assumptions, one can check it on the infinitesimal level, as seen from the next
proposition.

\begin{proposition}\label{prop5.C}
Assume that:
\begin{itemize}
\item  $E'$ and $E$ are two finite or countably infinite sets;

\item $\La:E'\dasharrow E$ is a stochastic Feller matrix with finitely many
nonzero entries in every row;

\item $Q'$ and $Q$ are two matrices of format $E'\times E'$ and $E\times E$,
respectively, satisfying the assumptions of Theorem\/ {\rm\ref{thm2.1};}

\item $\{T'(t)\}$ and $\{T(t)\}$ are the corresponding Feller semigroups
afforded by that theorem.

\end{itemize}

Then  $Q'\La=\La Q$ implies that $T'(t)\La=\La T(t)$ for all $t\ge0$.
\end{proposition}

Note that the assumptions on $\La$, $Q'$, and $Q$ imply that the products
$Q'\La$ and $\La Q$ are well defined and, moreover, these two matrices have
finitely many nonzero entries in every row.

\begin{proof}
See \cite[Section 6.2]{BO-JFA12}.
\end{proof}

We will use this result to check condition \eqref{eq5.4} from Proposition
\ref{prop5.B}.

\subsection{Approximation of semigroups}\label{sect5.A}
Here we are going to show that, under suitable additional assumptions, the
boundary semigroup $T_\infty(t)$ that is afforded by the construction of
Proposition \ref{prop5.B} is approximated by  semigroups $T_r(t)$ in the sense
of Definition \ref{def3.B}.

We keep to the hypotheses of Proposition \ref{prop5.B}. Next, we assume that
Condition \ref{cond4.A} is satisfied and one more condition holds:

\begin{condition}\label{cond5.A}
For every $r\in I$, the space $C_c(E_r)$ of finitely supported functions is a
core for the generator $A_r$ of the semigroup $T_r(t)$. Moreover, this space is
invariant under the action of $A_r$.
\end{condition}

We set $\LLL_r=C_0(E_r)$, $\LLL_\infty=C_0(E_\infty)$. Given a function $f$ on
$E_\infty$, we define the function $\pi_r f$ on $E_r$ by
$$
(\pi_r f)(x):=f(\varphi_r(x)), \qquad x\in E_r.
$$
Since $\varphi_r(E_r)$ is assumed to be a locally finite subset of $E_\infty$
(see part (i) of Condition \ref{cond4.A}), $\pi_r$ maps $\LLL_\infty$ into
$\LLL_r$. Obviously, the norm of $\pi_r$ is less or equal to 1.

\begin{proposition}\label{prop5.D}
Under the above assumptions, $T_r(t)\to T_\infty(t)$ in the sense of Definition
\ref{def3.B}.
\end{proposition}

\begin{proof}
Let $A_\infty$ be the generator of the boundary semigroup $T_\infty(t)$. By
virtue of Theorem \ref{thm3.2}, it suffices to prove that the restriction of
$A_\infty$ to some core $\F$ is approximated by the operators $A_r$. As $\F$ we
take the linear span of the subspaces $\La^\infty_r C_c(E_r)\subset
C_0(E_\infty)$, where $r$ ranges over $I$. The second condition postulated
above says that $C_c(E_r)$ is a core of $A_r$; consequently, $\F$ is a core for
$A_\infty$, by virtue of Proposition \ref{prop5.A}.

According to Definition \ref{def3.C} we have to show that for any vector
$f\in\F$ one can find a family of vectors $f_r\in\dom(A_r)$ such that the
following two limit relations hold: $f_r\to f$ and $A_r f_r\to A_\infty f$ as
$r\to\infty$.

Without loss of generality we may assume that $f\in\La^\infty_{s}g$ with $g\in
C_c(E_{s})$ for some $s\in I$. Next, for $r>s$ we set $f_r:=\La^r_{s}g$ and
observe that it suffices to prove the first limit relation only. Indeed, once
we know that $f_r\to f$ with such a choice of $\{f_r\}$, the second limit
relation, $A_r f_r\to A_\infty f$, follows simply by replacing $g$ with
$A_{s}g$, because the links intertwine the generators. We also use the fact
that $g\in C_c(E_{s})$ implies $A_{s}g\in C_c(E_{s})$ (see the end of the
second condition above).

We proceed to the proof of the convergence $f_r\to f$. By Definition
\ref{def3.A}, it means that
$$
\lim_{r\to\infty}\sup_{x\in E_r}|f_r(x)-f(\varphi_r(x))|=0.
$$
Without loss of generality we may assume that $g$ is the delta-function at a
point $y\in E_{s}$, but then the desired limit relation holds by virtue of
Condition \ref{cond4.A}.
\end{proof}

\section{A toy example: the one-dimensional Laguerre diffusion}\label{sect6}

In this section we apply the abstract formalism described above to a
construction of the Laguerre diffusion process on the halfline $\R_+$,
generated by the differential operator
\begin{equation*}
x\frac{d^2}{d x^2}+(c-x)\frac{d}{dx}
\end{equation*}
(here $c>0$ is a parameter). This process is well known --- it is related to
the Bessel process in the same way as the Ornstein-Uhlenbeck process is related
to the Wiener process, see, e.g. Eie \cite{Eie83}. Thus, the final result is by
no means new. However, the detailed exposition presented below will serve us as
a preparation and a guiding example for Section \ref{sect8}, where we establish
the main results.

\subsection{The binomial projective system $\mathbb B$}

Recall that $\mathbb B$ was introduced in Section \ref{sect4.A}. We will prove
two technical propositions concerning the properties of the links of $\mathbb
B$.

Observe that every link $\LB^{r'}_r$ can be applied to an arbitrary function on
$\Z_+$ (viewed as a column vector), because each row in $\LB^{r'}_r$ has
finitely many nonzero entries. As for $\LB^\infty_r$, it can be applied to
functions on $\Z_+$ with moderate (say, at most polynomial) growth at infinity.
In the next two propositions we provide explicit formulas for the action of the
links on functions of some special kind.

Introduce a notation:
$$
y^{\down m}=y(y-1)\dots(y-m+1), \qquad m\in\Z_+.
$$
Here $y$ is assumed to range over $\R_+$ or $\Z_+$, depending on the context.
By $\mathbf1_m$, where $m\in\Z_+$,  we denote the function on $\Z_+$ equal to 1
at $m$ and to 0 on $\Z_+\setminus\{m\}$. The letter $q$ always denotes a number
from the open interval $(0,1)$. Note that
\begin{equation}\label{eq6.13}
\lim_{q\to+0}\frac1{m!q^m}q^y y^{\down m}=\mathbf1_m \qquad \textrm{on $\Z_+$}.
\end{equation}

\begin{proposition}\label{prop6.A}
Assume $r'>r>0$ and\/ $0<q<1$, and let $l$ range over $\Z_+$. Regard
$\LB^{r'}_{r}$ as an operator  in the space of functions on $\Z_+$ transforming
a function $F(l)$ to a function $G(l)$. Under this transformation
\begin{gather}
l^{\down m}\mapsto\left(\frac{r}{r'}\right)^m l^{\down m} \label{eq1}\\
\mathbf1_m \mapsto \frac1{m!}\left(\frac{1-q'}{q'}\right)^m\cdot
(q')^l l^{\down m}, \qquad q':=1-\frac{r}{r'} \label{eq2}\\
 q^l l^{\down m} \mapsto\left(\frac{qr}{q'r'}\right)^m
(q')^l l^{\down m}, \qquad q':=1-(1-q)\frac{r}{r'} \label{eq3}
\end{gather}
\end{proposition}

\begin{proof}
Let us prove \eqref{eq3}. The function $F(l)=q^l l^{\down m}$ vanishes on
$\{0,\dots,m-1\}$ and the same holds for $\LB^{r'}_{r}F$, because the matrix
$\LB^{r'}_{r}$ is lower triangular. Therefore, it suffices to compute
$(\LB^{r'}_{r}F)(l)$ for $y\ge m$. We have
$$
(\LB^{r'}_{r}F)(l)=\sum_{k=m}^l\left(1-\frac{r}{r'}\right)^{l-k}
\left(\frac{r}{r'}\right)^k\, \frac{l!}{(l-k)!\,k!}\,q^k\frac{k!}{(k-m)!}.
$$
Setting $k'=k-m$ and $l'=l-m$ we rewrite the right-hand side as
$$
\left(\frac{qr}{r'}\right)^m l^{\down m}
\sum_{k'=0}^{l'}\left(1-\frac{r}{r'}\right)^{l'-k'}
\left(\frac{qr}{r'}\right)^{k'}\, \frac{l'!}{(l'-k')!\,k'!}.
$$
The latter sum equals
$$
(q')^{l'}=(q')^{-m}(q')^l,
$$
which leads to the desired result.

Formula \eqref{eq1} can be checked in exactly the same way. Observe also that
\eqref{eq1} is a limit case of \eqref{eq3} as $q\to1$.

Formula \eqref{eq2} is immediate from the very definition of $\LB^{r'}_{r}$. On
the other hand, \eqref{eq2} can also be obtained from \eqref{eq3} as a limit
case: to see this, divide by $m!q^m$, let $q\to0$ and use \eqref{eq6.3}.
\end{proof}

\begin{proposition}\label{prop6.B}
Assume $r>0$ and\/ $0<q<1$, and let $l$ range over $\Z_+$ while $x$ ranges over
$\R_+$. Regard $\LB^{\infty}_{r}$ as an operator transforming a function $F(l)$
on $\Z_+$ to a function $G(x)$ on $\R_+$. Under this transformation
\begin{gather}
l^{\down m}\,\mapsto\, r^{m} x^m  \label{eq4}\\
\mathbf1_m\, \mapsto\, \frac{r^m}{m!}\cdot q_\infty^x x^m, \qquad
q_\infty:=e^{-r}
\label{eq5}\\
q^l l^{\down m}\,\mapsto\, q^m r^{m}\cdot q_\infty^x x^m, \qquad
q_\infty:=e^{-(1-q)r}. \label{eq6}
\end{gather}
\end{proposition}

\begin{proof} We may argue exactly as in the proof of Proposition \ref{prop6.A},
replacing the binomial distribution by the Poisson distribution.

Alternatively, one can use \eqref{eq4.C} and pass to the limit $l\to\infty$,
$r'\to\infty$, $l/r'\to x$ in the formulas of Proposition \ref{prop6.A}.
\end{proof}

\subsection{The Meixner and Laguerre semigroups}

Introduce a $Q$-matrix of format $\Z_+\times\Z_+$, depending on parameters
$c>0$ and $r>0$, with the entries
\begin{gather*}
\Qc_r(k,k+1)=r(c+k), \qquad \Qc_r(k,k-1)=(r+1)k,\\
\Qc_r(k,k)=-[r(c+k)+(r+1)k]=-[(2r+1)k+rc],\\
\Qc_r(k,k')=0, \qquad |k-k'|\ge2,
\end{gather*}
where $k$ ranges over $\Z_+$. Let us regard $\Qc_r$ as a difference operator
acting on functions on $\Z_+$, which are interpreted as column vectors:
\begin{equation}\label{eq6.4}
(\Qc_r F)(l)=r(c+l)F(l+1)+(r+1)l F(l-1)-[(2r+1)l+rc]F(l),
\end{equation}
where $l\in\Z_+$. As is seen from the next proposition, this difference
operator is related to the classical {\it Meixner orthogonal polynomials\/}.
Recall the definition of the these polynomials (see, e.g.,
Koekoek--Lesky--Swarttouw \cite{KLS10} and references therein):

The Meixner polynomials are orthogonal  with respect to the negative binomial
distribution on $\Z_+$,
$$
\sum_{l\in\Z_+}(1+r)^{-c}\frac{(c)_l}{l!}\,\left(\frac r{1+r}\right)^l \de_l,
$$
where $(c)_l:=c(c+1)\dots(c+l-1)$ is the Pochhammer symbol and $\de_l$ denotes
the delta measure at $l$. The explicit expression for the {\it monic\/} Meixner
polynomial of degree $n=0,1,2,\dots$ is
\begin{equation}\label{eq6.3}
\Meix_n(l; c,r)=(c)_n\sum_{m=0}^n(-r)^{n-m}\frac{n^{\down m}} {(c)_m
m!}l^{\down m}.
\end{equation}

\begin{proposition}\label{prop6.D}
The Meixner difference operator \eqref{eq6.4} preserves the space of
polynomials. We have
\begin{equation}\label{eq6.E}
\Qc_r:  l^{\down m}\to -m l^{\down m}+rm(m+c-1) l^{\down (m-1)}
\end{equation}
and
\begin{equation}\label{eq6.5}
\Qc_r: \Meix_n(l; c,r)\to -n \Meix_n(l;c,r).
\end{equation}
\end{proposition}

Thus, the Meixner difference operator is diagonalized in the basis of the
Meixner polynomials.

\begin{proof}
All claims can be verified directly. For \eqref{eq6.5}, see also \cite{KLS10}.
\end{proof}

\begin{proposition}\label{prop6.E}
For arbitrary  $r'>r>0$, we have
$$
\Qc_{r'}\,\LB^{r'}_{r}=\LB^{r'}_{r}\Qc_{r}.
$$
\end{proposition}

\begin{proof}
Because the $Q$-matrices in question have a simple tridiagonal form and the
entries of $\LB^{r'}_{r}$ are given by a simple expression, a direct check is
possible. However, we prefer to give another proof, which has the advantage of
being more conceptual and well suited for the generalization that we need.

Observe that
\begin{equation*}
\Qc_{r'}\,\LB^{r'}_{r} F=\LB^{r'}_{r} \Qc_{r} F
\end{equation*}
for any polynomial $F$. Indeed, it suffices to check this for
$F=\Meix_m(\,\cdot\,;c,r)$. It follows from \eqref{eq1} and \eqref{eq6.3} that
\begin{equation*}
\LB^{r'}_{r} \Meix_m(\,\cdot\,;c,r)= \left(\frac {r}{r'}\right)^m
\Meix_m(\,\cdot\,; c,r'),
\end{equation*}
and then we use \eqref{eq6.5} to conclude that both $\Qc_{r'}\,\LB^{r'}_{r}$
and $\LB^{r'}_{r}\Qc_{r}$ multiply $\Meix_m(\,\cdot\,;c,r)$ by $-m(r/r')^m$.

Further, both matrices $\Qc_{r'}\,\LB^{r'}_{r}$ and $\LB^{r'}_{r} \Qc_{r}$ have
finitely many nonzero entries in every row. Since the polynomials separate
points on $\Z_+$, these two matrices coincide.
\end{proof}

\begin{proposition}\label{prop6.F}
For any $c,r>0$, the matrix  $\Qc_r$ satisfies the assumptions of Theorem
\ref{thm2.1} with functions $\ga(k)=\eta(k)=k+1$.
\end{proposition}

\begin{proof}
Easy direct check.
\end{proof}

This proposition makes it possible to apply Theorem \ref{thm2.1}, which in turn
entails the following assertions.

\begin{corollary}\label{cor6.A}
{\rm(i)} The $Q$-matrix $\Qc_r$ gives rise to a Feller semigroup $T^{(c)}_r(t)$
on $C_0(\Z_+)$ whose generator $A^{(c)}_r$ is implemented by $\Qc$.

{\rm(ii)} The subspace $C_0(\Z_+)$ is a core for generator $A^{(c)}_r$.
\end{corollary}

We call $\Tc_r(t)$ the {\it Meixner semigroup\/}. It determines a continuous
time Markov chain on $\Z_+$ which we call the  {\it Meixner chain\/} and denote
by $\Xc(t)$.

\begin{proposition}\label{prop6.K}
For every $c>0$ there exists a unique Feller Markov process $X^{(c)}(t)$ on
$\R_+$ such that the corresponding Feller semigroup, denoted by $\Tc(t)$, is
consistent with the Meixner semigroups $\Tc_r(t)$, $r>0$, in the sense that
\begin{equation*}
\Tc(t)\,\LB^{\infty}_{r}=\LB^{\infty}_{r}\,\Tc_{r}(t),\qquad t\ge 0, \quad r>0.
\end{equation*}
\end{proposition}

\begin{proof}
We know that the $Q$-matrices with various values of parameter $r$ are
consistent with the links (Proposition \ref{prop6.E}). It follows, by virtue of
Proposition \ref{prop5.C}, that the semigroups are also consistent with the
links. Therefore, we may apply Proposition \ref{prop5.B}, which gives the
desired result.
\end{proof}

We call $X^{(c)}(t)$ and $\Tc(t)$ the {\it Laguerre process\/} and the {\it
Laguerre semigroup\/}, respectively; this terminology is justified by the
results of Section \ref{sect6.A}.

\subsection{A family of cores for Markov semigroup generators}

For any fixed $q\in(0,1)$, the functions $q^x x^m$, $m=0,1,2,\dots$, span a
dense subspace in $C_0(\R_+)$, see Borodin--Olshanski \cite[Corollary
3.1.6]{BO-MMJ13}. This also implies that the functions $q^l l^m$, where
$m=0,1,2,\dots$ and $l$ ranges over $\Z_+$, span a dense subspace in
$C_0(\Z_+)$. These facts are used in the next proposition.

\begin{proposition}\label{prop6.G}
{\rm(i)} For any $r'>r>0$, the operator $\LB^{r'}_{r}: C_0(\Z_+)\to C_0(\Z_+)$
has a dense range.

{\rm(ii)} Likewise, for any $r>0$, the operator $\LB^\infty_r: C_0(\Z_+)\to
C_0(\R_+)$ has a dense range.
\end{proposition}

\begin{proof}
(i) Take an arbitrary $q\in(0,1)$. By \eqref{eq3}, $\LB^{r'}_{r}$ maps the
linear span of functions $q^l l^m$, $m=0,1,2,\dots$ onto the linear span of
functions $(q')^l l^m$, with some other $q'\in(0,1)$, see \eqref{eq3}. Since
these spans are dense, we get the desired claim.

(ii) The same argument, with reference to \eqref{eq6}.
\end{proof}

Recall that $\Ac_r$ denotes the generator of semigroup $\Tc_r(t)$ (Corollary
\ref{cor6.A}). Likewise, let $\Ac$ denote the generator of semigroup $\Tc(t)$.

\begin{proposition}\label{prop6.H}
Fix an arbitrary number $q\in(0,1)$.

{\rm(i)} For every $r>0$, the linear span of functions $q^l l^m$,
$m=0,1,2,\dots$, where argument $l$ ranges over\/ $\Z_+$, is a core for
$\Ac_r$.

{\rm(ii)} Likewise, the linear span of functions $q^x x^m$, $m=0,1,2,\dots$,
where argument $x$ ranges over\/ $\R_+$, is a core for $\Ac$.
\end{proposition}

\begin{proof}
(i) Observe that if $r_2>r_1>0$ and $\F_1$ is a core for $\Ac_{r_1}$, then
$\F_2:=\LB^{r_2}_{r_1}\F_1$ is a core for $\Ac_{r_2}$. Indeed, by virtue of
claim (i) of Proposition \ref{prop6.G}, we may apply the argument of
Proposition \ref{prop5.A}.

Now take $r_2=r$ and $r_1=(1-q)r$. Then, as seen from \eqref{eq2}, the linear
span of functions $q^y y^m$ is just the image under $\LB^{r_2}_{r_1}$ of the
space $C_c(\Z_+)$. By virtue of Proposition \ref{prop6.F} and claim (iv) of
Theorem \ref{thm2.1}, $C_c(\Z_+)$ is a core for $\Ac_{r_1}$. Therefore, its
image is a core for $\Ac_{r_2}$.

(ii) We argue as above. First, application of claim (ii) of Proposition
\ref{prop6.G} allows us to conclude that if $\F\subset C_0(\Z_+)$ is a core for
$\Ac_r$ for some $r>0$, then $\LB^\infty_r\F$ is a core for $\Ac$.

Next, given $q\in(0,1)$ we take $r=-\log q$ and $\F=C_0(\Z_+)$. As mentioned
above, $\F$ is a core for $\Ac_r$. On the other hand, \eqref{eq5} shows that
the linear span of functions $q^x x^m$ coincides with $\LB^\infty_r \F$.
\end{proof}

\subsection{The Laguerre differential operator}\label{sect6.A}

Proposition \ref{prop6.H} implies that the generator $\Ac$ is uniquely
determined by its action on functions of the form $q^x x^m$, $m=0,1,2,\dots$,
with an arbitrary fixed $q\in(0,1)$. This action can be readily computed from
the basic relation $\Ac\,\LB^\infty_r=\LB^\infty_r\Qc_r$:

\begin{proposition}\label{prop6.I}
The action of $\Ac$ on functions of the form $q^x x^m$, $m=0,1,2,\dots$ is
implemented by the differential operator
\begin{equation} \label{eq6.6}
\Dc:=x\frac{d^2}{d x^2}+(c-x)\frac{d}{dx}\,.
\end{equation}
\end{proposition}

\begin{proof}
Let $r>0$ be related to $q\in(0,1)$ by $r=-\log q$. Consider the functions
$$
f_m(x):=\frac{r^m}{m!}q^x x^m=\frac{r^m}{m!}e^{-rx} x^m, \qquad m=0,1,2,\dots,
\quad x\in\R_+.
$$
By \eqref{eq5},
$$
\LB^\infty_r \mathbf1_m=f_m, \qquad m=0,1,2,\dots\,.
$$

On the other hand, it is directly verified that the difference operator $\Qc_r$
defined in \eqref{eq6.4} acts on the delta functions $\mathbf1_m$ in the same
way as the differential operator $\Dc$ acts on the functions $f_m$:
\begin{gather}
\Qc_r
\mathbf1_m=r(c+m-1)\mathbf1_{m-1}+(r+1)(m+1)\mathbf1_{m+1}-[(2r+1)m+rc]\mathbf1_m
\label{eq6.A}\\
\Dc f_m=r(c+m-1)f_{m-1}+(r+1)(m+1)f_{m+1}-[(2r+1)m+rc]f_m, \label{eq6.B}
\end{gather}
where
$$
\mathbf1_{-1}:=0, \qquad f_{-1}:=0.
$$

This concludes the proof.
\end{proof}

Consider the gamma distribution on $\R_+$ with parameter $c$:
\begin{equation*}
\frac1{\Ga(c)}x^{c-1}e^{-x}dx, \quad x\in\R_+,
\end{equation*}
and let $\Lag_n(x;c)$ denote the {\it monic Laguerre polynomials\/} of degree
$n=0,1,2,\dots$, which are orthogonal with respect to this distribution:
\begin{equation}\label{eq6.C}
\Lag_n(x;c)=(c)_n\sum_{m=0}^n(-1)^{n-m}\frac{n^{\down m}}{(c)_m m!}x^m.
\end{equation}

The differential operator $\Dc$ is diagonalized in the basis of the Laguerre
polynomials:
\begin{equation}\label{eq6.7}
\Dc \Lag_n(\,\cdot\,;c)=-n\Lag_n(\,\cdot\,;c), \qquad n=0,1,2,\dots\,.
\end{equation}
Note also that
\begin{equation}\label{eq6.14}
\LB^\infty_r  \Meix_n(\,\cdot\,;c,r)=r^n \Lag_n(\,\cdot\,; c).
\end{equation}
The proof is immediate: we compare the expansions of the Meixner and Laguerre
polynomials in the bases $\{l^{\down m}\}$ and $\{x^m\}$, respectively  (see
\eqref{eq6.3} and \eqref{eq6.C}), and then apply \eqref{eq4}, which says that
$\LB^\infty_r$ takes the factorial monomial $l^{\down m}$ to $r^mx^m$.

\subsection{Approximation}

We use the embedding $\varphi_r:\Z_+\to\R_+$ introduced in Example \ref{ex4.A}
and define the projection $\pi_r: C_0(\R_+)\to C_0(\Z_+)$ as in Section
\ref{sect5.A}.

\begin{proposition}\label{prop6.J}
Let $c>0$ be fixed. As $r\to+\infty$, the Meixner semigroups $\Tc_r(t)$
approximate the Laguerre semigroup $\Tc(t)$ in the sense of Definition
\ref{def3.B}.
\end{proposition}

\begin{proof}
Let us check all the hypotheses of Proposition \ref{prop5.D}. Then the desired
result will follow from that proposition.

In fact, the assumptions stated in Section \ref{sect5.B} are satisfied: we know
that $\mathbb B$ is a Feller system, the Meixner semigroups are consistent with
the links of $\mathbb B$, those are  Feller links, and, by the very definition,
the Laguerre semigroup is the boundary semigroup determined by the Meixner
semigroups.

Next, the fulfilment of Condition \ref{cond4.A} was established in Example
\ref{ex4.A}.

It remains to check Condition \ref{cond5.A}. In our situation, it consists in
the requirement that $C_c(\Z_+)$ is a core for generator $A^{(c)}_r$ and,
moreover, is invariant under its action. The fact that $C_c(\Z_+)$ is a core
follows from Corollary \ref{cor6.A}, item (ii). Its invariance follows from
item (i), because $C_c(\Z_+)$ is obviously invariant under the action of
$\Qc_r$.

This completes the proof.
\end{proof}

\section{A few definitions}\label{sect7}

Here we collect some basic definitions that will be needed in the next section.
For a more detailed information we refer to Sagan \cite{Sagan} and Stanley
\cite{Stanley} (generalities on Young diagrams, Young tableaux, and symmetric
functions); Olshanski--Regev--Vershik \cite{ORV-Birkh03} (Frobenius--Schur
symmetric functions); Borodin--Olshanski \cite{BO-MMJ13}, \cite{BO-PTRF09}
(Thoma's simplex and Thoma's cone).

\subsection{Young diagrams}

Recall that the {\it Young poset\/} is the set $\Y$ of all Young diagrams
(including the empty diagram $\varnothing$) with the partial order determined
by containment of one Young diagram in another. For $\la\in\Y$ we denote by
$|\la|$ the number of boxes of $\la$ and we set
$$
\Y_n=\{\la\in\Y: |\la|=n\}, \qquad n=0,1,2,\dots\,.
$$
This makes $\Y$ a {\it graded poset\/}. It is actually a lattice, so it is
often called the {\it Young lattice\/}.

The {\it dimension\/} of a diagram $\la\in\Y$, denoted by $\dim\la$, is the
number of standard Young tableaux of shape $\la$, which is the same as the
number of saturated chains
$$
\varnothing=\la^{(0)}\subset \la^{(1)}\subset\dots\subset \la^{(n)}=\la, \qquad
n:=|\la|,
$$
in the poset $\Y$.

More generally, for arbitrary two diagrams $\mu,\la\in\Y$ we define
$\dim(\mu,\la)$ as the number of standard Young tableaux of skew shape
$\la/\mu$ provided that $\mu\subseteq\la$; otherwise $\dim(\mu,\la)=0$ (let us
agree that $\dim(\la,\la)=1$). Obviously, $\dim\la=\dim(\varnothing,\la)$. If
$\mu\subset\la$, then $\dim(\mu,\la)$ equals the number of saturated chains
with ends $\mu$ and $\la$.

\subsection{Symmetric functions}

By $\Sym$ we denote the graded algebra of symmetric functions over the base
field $\R$. We will need two bases in $\Sym$, both indexed by arbitrary
diagrams $\mu\in\Y$: the {\it Schur functions\/} $S_\mu$ and the {\it
Frobenius--Schur functions\/} $\FS_\mu$. The relationship between $S_\mu$'s and
$\FS_\mu$'s is similar to the relationship between the one-variate monomials
$x^m$ and their factorial counterparts $x^{\down m}$. Observe that $x^{\down
m}$ can be characterized as a unique polynomial in $x$ with highest degree term
$x^m$ and such that it vanishes at the integer points $0,1\dots,m-1$. Likewise,
one can realize $\Sym$ as a subalgebra in $\Fun(\Y)$, the algebra of
real-valued functions on $\Y$ with all operations defined pointwise; see the
next two paragraphs. Then $\FS_\mu$ can be characterized as a unique element of
$\Sym$ that has top degree term $S_\mu$ and vanishes at all diagrams strictly
contained in $\mu$.

Let $p_1,p_2,\dots$ denote the power-sum symmetric functions. We turn them into
functions on $\Y$ by setting
$$
p_k(\la):=\sum_{i=1}^\infty \left((\la_r-i+\tfrac12)^k-(-i+\tfrac12)^k\right)
=\sum_{i=1}^d\left(a_r^k+(-1)^{k-1}b_r^k\right),
$$
where $\la$ ranges over $\Y$, $(\la_1,\la_2,\dots)$ is the partition
corresponding to $\la$, $d$ is the number of boxes on the main diagonal of
$\la$, and $(a_1,\dots,a_d;b_1,\dots,b_d)$ is the collection of the {\it
modified Frobenius coordinates\/} of $\la$:
$$
a_r=\la_r-i+\tfrac12, \quad b_r=\la'_r-i+\tfrac12, \qquad i=1,\dots,d
$$
(here $\la'$ is the transposed diagram). One can easily prove that the
resulting functions remain algebraically independent.

Next, every element $F\in\Sym$ is uniquely written as a polynomial in
$p_1,p_2,\dots$; then we define $F(\la)$ as the same polynomial in numeric
variables $p_1(\la), p_2(\la),\dots$\,. In this way we get the desired
embedding of $\Sym$ into $\Fun(\Y)$.

A fundamental property of the Frobenius--Schur functions is the following
identity (see \cite[Section 2]{ORV-Birkh03}) relating them to the dimension
function in the poset $\Y$:
\begin{equation}\label{eq7.A}
l^{\down m}\frac{\dim(\mu,\la)}{\dim \la}=\FS_\mu(\la), \qquad l:=|\la|, \quad
m:=|\mu|.
\end{equation}

\subsection{The Thoma simplex and the Thoma cone}

The {\it Thoma simplex\/} is the subspace $\Om$ of the infinite product space
$\R_+^\infty\times\R_+^\infty$ formed by all couples $(\al,\be)$, where
$\al=(\al_i)$ and $\be=(\be_i)$ are two infinite sequences such that
\begin{equation}\label{eq20}
\al_1\ge\al_2\ge\dots\ge0, \qquad \be_1\ge\be_2\ge\dots\ge0
\end{equation}
and
\begin{equation}\label{eq21}
\sum_{i=1}^\infty\al_i+\sum_{i=1}^\infty\be_i\le 1.
\end{equation}
We equip $\Om$ with the product topology inherited from
$\R_+^\infty\times\R_+^\infty$. Note that in this topology, $\Om$ is a compact
metrizable space.

The {\it Thoma cone\/} $\wt\Om$ is the subspace of the infinite product space
$\R_+^\infty\times\R_+^\infty\times\R_+$ formed by all triples
$\om=(\al,\be,\de)$, where $\al=(\al_i)$ and $\be=(\be_i)$ are two infinite
sequences and $\de$ is a nonnegative real number, such that the couple
$(\al,\be)$ satisfies \eqref{eq20} and the modification of the inequality
\eqref{eq21} of the form
$$
\sum_{i=1}^\infty\al_i+\sum_{i=1}^\infty\be_i\le \de.
$$
We set $|\om|=\de$.

Note that $\wt\Om$ is a locally compact space in the product topology inherited
from $\R_+^\infty\times\R_+^\infty\times\R_+$. The space $\wt\Om$ is also
metrizable and has countable base. Every subset of the form $\{\om\in\wt\Om:
|\om|\le\const\}$ is compact. Therefore, a sequence of points $\om_n$ goes to
infinity in $\wt\Om$ if and only if $|\om_n|\to\infty$.

We will identify $\Om$ with the subset of $\wt\Om$ formed by triples
$\om=(\al,\be,\de)$ with $\de=1$. The name ``Thoma cone'' given to $\wt\Om$ is
justified by the fact that $\wt\Om$ may be viewed as the cone with the base
$\Om$: the ray of the cone passing through a base point $(\al,\be)\in\Om$
consists of the triples $\om=(r\al,r\be,r)$, $r\ge0$.

More generally, for $\om=(\al,\be,\de)\in\wt\Om$ and $r>0$ we set
$r\om=(r\al,r\be,r\de)$.

\subsection{The embeddings $\Sym\to \Fun(\wt\Om)$ and
$\Y\to\wt\Om$}\label{sect7.A}

We embed $\Sym$ into the algebra of (non necessarily bounded) continuous
functions on the Thoma cone by setting
$$
p_k(\om)=\begin{cases}\sum_{i=1}^\infty \al_i^k+ (-1)^{k-1}\sum_{i=1}^\infty
\be_i^k, & k=2,3,\dots\\
|\om|, & k=1,
\end{cases}
$$
where $\om$ ranges over $\wt\Om$.

We embed the set $\Y$ into $\wt\Om$ through the map
$$
\la\mapsto \om_\la:=((a_1,\dots,a_d,0,0,\dots), \, (b_1,\dots,b_d,0,0,\dots),
\, |\la|),
$$
where, as above, $(a_1,\dots,a_d;b_1,\dots,b_d)$ is the collection of the
modified Frobenius coordinates of a diagram $\la\in\Y$. Note that
$|\om_\la|=|\la|$.

For any $F\in\Sym$, the restriction of the function $F(\om)$ to the subset
$\Y\subset\wt\Om$ agrees with the previous definition of the function $F(\la)$:
$$
F(\om_\la)=F(\la), \qquad \la\in\Y.
$$

\section{Construction of Feller processes on the Thoma cone}\label{sect8}

\subsection{The projective system associated with the Young bouquet}
The representation theory of inductive limit groups provides two fundamental
examples of projective systems. One is related to the infinite symmetric group
$S(\infty)$ and comes from the Young graph $\Y$, and the other one is related
to the infinite-dimensional unitary group $U(\infty)$ and comes from the
Gelfand--Tsetlin graph $\GT$. The boundaries of these two projective systems
can be viewed as dual objects to $S(\infty)$ and $U(\infty)$, respectively. In
attempt to explain a surprising similarity between the two boundaries, we
introduced in \cite{BO-MMJ13} a new object which serves as a ``mediator''
between $\Y$ and $\GT$. We called it the {\it Young bouquet\/}; it is a close
relative of $\Y$ and at the same time it can be obtained as a degeneration of
$\GT$. Associated with the Young bouquet is a new projective system denoted by
$\YB$. Because $\GT$ is graded by discrete set $\Z_+$, the associated
projective system has $\Z_+$ as its index set, but under degeneration the index
set becomes continuous. Here is a formal definition of $\YB$:

The index set of the projective system $\YB$ is the set $\R_{>0}$ and each set
$E_r$ is a copy of the set $\Y$. For every couple $r'>r$ of positive real
numbers, the corresponding link $\Y\dasharrow\Y$ is the following stochastic
matrix of format $\Y\times\Y$:
\begin{equation}\label{eq8.R}
\LYB^{r'}_r(\la,\mu)=\left(1-\frac{r}{r'}\right)^{l-m}\left(\dfrac{r}{r'}\right)^m\,
\dfrac{l!}{(l-m)!\,m!}\,\dfrac{\dim\mu\, \dim(\mu,\la)}{\dim\la},
\end{equation}
where $l:=|\la|$ and $m:=|\mu|$.

Note that \eqref{eq8.R} factorizes into a product of two links, which refer to
two projective systems, the binomial system $\mathbb B$ and the Young graph
$\Y$:
\begin{equation}\label{eq8.S}
\LYB^{r'}_r(\la,\mu)=\LB^{r'}_r(l,m)\,\LY^l_m(\la,\mu),
\end{equation}
where
\begin{equation}\label{eq8.T}
\LY^l_m(\la,\mu):=\dfrac{\dim\mu\, \dim(\mu,\la)}{\dim\la}.
\end{equation}

The links \eqref{eq8.R} satisfy the relation
$$
\LYB^{r''}_{r'}\,\LYB^{r'}_r=\LYB^{r''}_r, \qquad r''>r'>r,
$$
so that they do determine a projective system. We refer to \cite{BO-MMJ13} for
more details.

By \cite[Theorem 3.4.7]{BO-MMJ13}, the boundary of $\YB$ is the Thoma cone
$\wt\Om$ together with a family of links $\wt\Om\dasharrow\Z_+$ indexed by
positive real numbers $r$ and given by
\begin{equation}\label{eq8.2}
\LYB^\infty_r(\om, \mu)=e^{-r|\om|}\frac{r^m}{m!}\,\dim\mu\cdot S_\mu(\om),
\qquad \om\in\wt\Om, \quad \mu\in\Y.
\end{equation}
Recall that $S_\mu$ is the Schur symmetric function and its value at
$\om\in\wt\Om$ is understood in accordance with the definition given in Section
\ref{sect7.A}.

\begin{proposition}\label{prop8.C}
The projective system $\YB$ is Feller in the sense of the definition given in
Section \ref{sect4.A}.
\end{proposition}

\begin{proof}
The links $\LYB^{r'}_{r}$ and $\LYB^\infty_r$ are Feller: this immediately
follows from the Feller property of the links $\LB^{r'}_{r}$ and
$\LB^\infty_r$. It remains to show that the product topology of the space
$\wt\Om$ coincides with that defined by all the maps
$\om\mapsto\La^\infty_r(\om, \mu)$, where $r$ ranges over $\R_{>0}$ and $\mu$
ranges over $\Y$. Actually, this holds even if $r$ is any fixed number $>0$,
and the argument is similar to that given in Section \ref{sect4.A}.

Namely, we extend the above maps to the one-point compactification
$\wt\Om\cup\infty$ of $\wt\Om$ in a natural way: the value at infinity is equal
to 0 for any $\mu$, which agrees with the Feller property of the links. Then we
only have to check that any point of $\wt\Om\cup\infty$ is uniquely determined
by its images under the (extended) maps $\LYB^\infty_r(\,\cdot\,,\mu)$, where
$\mu$ ranges over $\Y$.

To do this, assume first that $\om\in\wt\Om$ and recall \eqref{eq8.2}. Keeping
$m$ fixed and summing the quantity in the right-hand side over $\mu\in\Y_m$ we
get
$$
e^{-rx}\frac{(rx)^m}{m!}, \qquad x:=|\om|,
$$
because
$$
\sum_{\mu\in\Y_m}\dim\mu\,S_\mu=(p_1)^m
$$
and $p_1(\om)=|\om|=x$.

Observe that for $r>0$ fixed, the quantities $e^{-rx}\frac{(rx)^m}{m!}$, where
$m$ ranges over $\Z_+$, determine $x$ uniquely. It follows, in particular, that
we can recognize whether we are dealing with an element of the Thoma cone
$\wt\Om$ or the added point $\infty$, because the latter case corresponds to
$x=+\infty$.

Therefore, it suffices to check that an element $\om\in\wt\Om$ is uniquely
determined by the quantities $S_\mu(\om)$, where $\mu$ ranges over $\Y$. But
this follows from the fact that the functions $p_1(\om), p_2(\om), \dots$
separate the points of the Thoma cone.
\end{proof}

Note that $\LYB^{r'}_r(\nu,\mu)$ vanishes unless $m\le n$ and
$\mu\subseteq\la$. This implies that each row of the matrix $\LYB^{r'}_r$ has
finitely many nonzero entries, so that the link can be applied to an arbitrary
function on $\Y$.

Below we denote by $\mathbf1_\mu$ the delta function on $\Y$ concentrated at
the point $\mu$, that is,
$$
\mathbf1_\mu(\la)=\begin{cases} 1, & \la=\mu, \\ 0, & \la\ne\mu.\end{cases}
$$

\begin{proposition}[cf. Proposition \ref{prop6.A}]\label{prop8.A}
Assume that:
\begin{itemize}
\item $r'>r>0$ and\/ $0<q<1$;

\item $\la$ range over $\Y$ and $l:=|\la|$;

\item $\mu\in\Y$ is fixed and $m=|\mu|$.

\end{itemize}

Regard\/ $\LYB^{r'}_{r}$ as linear map $F\mapsto G$ transforming a function
$F(\la)$ on $\Y$ to another function $G(\la)$. Under this transformation
\begin{gather}
\FS_\mu(\la)\, \mapsto\,\left(\frac{r}{r'}\right)^m \FS_\mu(\la), \label{eq8.A}\\
(\dim\mu)^{-1}\mathbf1_\mu\, \mapsto\,
\frac1{m!}\left(\frac{1-q'}{q'}\right)^m\cdot
(q')^l \FS_\mu(\la), \qquad q':=1-\frac{r}{r'}\,, \label{eq8.B}\\
q^l \FS_\mu(\la)\, \mapsto\,\left(\frac{qr}{q'r'}\right)^m (q')^l \FS_\mu(\la),
\qquad q':=1-(1-q)\frac{r}{r'}\,. \label{eq8.C}
\end{gather}
\end{proposition}

\begin{proof}
Let us prove \eqref{eq8.C}. The function $F(\la):=q^l \FS_\mu(\la)$ vanishes
unless $\la\supseteq\mu$,  and the same holds for $\LYB^{r'}_{r}F$, because the
matrix $\LYB^{r'}_{r}$ is lower triangular with respect to the partial order on
$\Y$ determined by the inclusion relation. Therefore, it suffices to compute
$(\LYB^{r'}_{r}F)(\la)$ for $\la\supseteq\mu$; in particular, $l\ge m$. We have
$$
(\LYB^{r'}_{r}F)(\la)
=\sum_{k=m}^l\LB^{r'}_{r}(l,k)\sum_{\ka\in\Y_k}\LY^l_k(\la,\ka)q^k\FS_\mu(\ka).
$$
For fixed $k$,
\begin{align*}
\LY^l_k(\la,\ka)q^k\FS_\mu(\ka)&=q^k\frac{\dim\ka\,\dim(\ka,\la)}{\dim\la}\,
k^{\down m}\,\frac{\dim(\mu,\ka)}{\dim\ka}\quad \textrm{by virtue of \eqref{eq7.A}}\\
&=q^k k^{\down m}\frac{\dim(\mu,\ka)\,\dim(\ka,\la)}{\dim\la},
\end{align*}
and summing the latter quantity over $\ka\in\Y_k$ gives
$$
q^k k^{\down m}\frac{\dim(\mu,\la)}{\dim\la}.
$$
Therefore,
\begin{align*}
(\LYB^{r'}_{r}F)(\la) &=k^{\down m}\frac{\dim(\mu,\la)}{\dim\la}\cdot
\sum_{k=m}^l\LB^{r'}_{r}(l,k)q^k\\
&=k^{\down
m}\frac{\dim(\mu,\la)}{\dim\la}\cdot\left(\frac{qr}{q'r'}\right)^m(q')^ll^{\down
m} \quad \textrm{by \eqref{eq3}}\\
&=\left(\frac{qr}{q'r'}\right)^m(q')^l\FS_\mu(\la),
\end{align*}
as desired.

Formula \eqref{eq8.A} can be checked in exactly the same way. Alternatively, it
can be obtained a limit case of \eqref{eq8.C} as $q\to1$.

Formula \eqref{eq8.B} is immediate from the very definition of $\LYB^{r'}_{r}$
and $\FS_\mu$. Alternatively, \eqref{eq8.B} can also be obtained from
\eqref{eq8.C} by a degeneration, like the derivation of \eqref{eq2} from
\eqref{eq3}, see the proof of Proposition \ref{prop6.A}.
\end{proof}

\begin{proposition}[cf. Proposition \ref{prop6.B}]\label{prop8.B}
Assume $r>0$ and\/ $0<q<1$; let $\la$ range over $\Y$ and $l=|\la|$; let $\om$
range over $\wt\Om$ and $x=|\om|$; let $\mu\in\Y$ be fixed and $m=|\mu|$.
Regard $\LYB^{\infty}_{r}$ as an operator transforming a function $F(\la)$ on
$\Y$ to a function $G(\om)$ on $\wt\Om$. Under this transformation
\begin{gather}
\FS_\mu(\la)\,\mapsto\, r^{m} S_\mu(\om),  \label{eq8.G}\\
(\dim\mu)^{-1}\mathbf1_\mu\, \mapsto\, \frac{r^m}{m!}\cdot q_\infty^x
S_\mu(\om),
\qquad q_\infty:=e^{-r},  \label{eq8.H}\\
q^l \FS_\mu(\la)\,\mapsto\, q^m r^{m}\cdot q_\infty^x S_\mu(\om), \qquad
q_\infty:=e^{-(1-q)r}. \label{eq8.I}
\end{gather}
\end{proposition}

\begin{proof} We may argue exactly as in the proof of the previous proposition.
\end{proof}

\subsection{Markov semigroups on $\Y$ and $\wt\Om$}\label{sect8.B}

We are going to introduce a $Q$-matrix of format $\Y\times\Y$ depending on the
triple $(z,z',r)$ of parameters, where $r>0$ and $(z,z')$, as usual, is subject
to Condition \ref{cond1.A}. For this we need some notation. Given $\la\in\Y$,
let $\la^+$ and $\la^-$ stand for the collections of boxes that can appended
to, respectively, removed from $\la$. For a box $\Box$, its {\it content\/} is
defined as the difference $c(\Box):=j-i$, where $i$ and $j$ are the row and
column numbers of $\Box$. The $Q$-matrix in question is denoted by $\Qz$ and
its non-diagonal entries $\Qz(\la,\ka)$, $\ka\ne\la$, vanish unless either
$\ka=\la+\Box$ or $\ka=\la-\Box$, meaning that $\ka$ is obtained from $\la$ by
appending a box $\Box\in\la^+$ or by removing a box $\Box\in\la^-$. In this
notation, the entries are given by
\begin{equation}\label{eq8.L}
\begin{aligned}
\Qz(\la,\la+\square)&=r(z+c(\square))(z'+c(\square))
\frac{\dim(\la+\square)}{(|\la|+1)\dim\la}, \quad
\square\in\la^+,\\
\Qz(\la,\la-\square)&=(r+1)\frac{|\la|\dim(\la-\square)}{\dim\la}, \quad \square\in\la^-,\\
-\Qz(\la, \la)&=(2r+1)|\la| +rzz'.
\end{aligned}
\end{equation}

Note that each row of $Q^\z_r$ has finitely many nonzero entries which sum to
0, and the constraints on the parameters imply that all off-diagonal entries
are nonnegative (in particular, $Q^\z_r(\la,\la+\Box)>0$ because of Condition
\ref{cond1.A}).

(For more detail about the definition of $\Qz$, we refer to Borodin--Olshanski
\cite{BO-PTRF06} and Olshanski \cite{Ols-IMRN12}. Formula \eqref{eq8.L}
coincides with that of \cite[Proposition 4.25]{Ols-IMRN12} and is a particular
case of \cite[(2.19)]{BO-PTRF06}. Note that parameter $\xi\in(0,1)$ from those
two papers is related to our parameter $r>0$ by $\xi=r(r+1)^{-1}$. In
\cite[(2.19)]{BO-PTRF06}, parameter $\xi$ may vary with time; our setup
corresponds to the particular case when $\xi$ is fixed, so that the time
derivative $\Dot\xi$ equals 0. Then formula \cite[(2.19)]{BO-PTRF06} simplifies
and reduces to \eqref{eq8.L}.)

We can interpret $\Qz$ as an operator in the vector space $\Fun(\Y)$ formed by
arbitrary real-valued functions on $\Y$:
\begin{equation}\label{eq8.K}
(\Qz F)(\la)=\sum_{\ka\in\Y}\Qz(\la,\ka)F(\ka), \qquad F\in\Fun(\Y).
\end{equation}

As explained in \cite{Ols-IMRN12}, this operator should be viewed as a
counterpart of the Meixner difference operator on $\Z_+$. The next step is to
introduce counterparts of the Meixner polynomials. According to
\cite[Definition 4.21]{Ols-IMRN12}, these are elements of $\Sym$ called the
{\it Meixner symmetric functions\/} and denoted by  $\MM_\nu$, where the index
$\nu$ ranges over $\Y$. They depend on the triple $(z, z',r)$ and are given by
the following expansion in the basis of the Frobenius--Schur symmetric
functions (cf. \eqref{eq6.3}):
\begin{equation}\label{eq8.D}
\begin{aligned}
\MM^{(z,z',r)}_\nu &=\sum_{\mu:\,\mu\subseteq\nu}
(-1)^{|\nu|-|\mu|}r^{|\nu|-|\mu|}
\frac{\dim\nu/\mu}{(|\nu|-|\mu|)!}\\
&\times \prod_{\square\in\nu/\mu}(z+c(\square))(z'+c(\square))\cdot \FS_\mu.
\end{aligned}
\end{equation}

\begin{proposition}[cf. Proposition \ref{prop6.D}]\label{prop8.D}
Under the action of $\Qz$ in $\Fun(\Y)$,
\begin{gather}
\FS_\mu \to
-|\mu|\FS_\mu+r\sum_{\Box\in\mu^-}(z+c(\Box))(z'+c(\Box))\FS_{\mu\setminus\Box},
\label{eq8.E}
\\
\MM^{(z,z',r)}_\mu\to -|\mu|\MM_\mu. \label{eq8.F}
\end{gather}
\end{proposition}

\begin{proof}
See \cite[Section 4.8]{Ols-IMRN12}.
\end{proof}

\begin{proposition}[cf. Proposition \ref{prop6.E}]\label{prop8.E}
For arbitrary $r'>r>0$, we have
$$
Q^{z,z'}_{r'}\,\LYB^{r'}_{r}=\LYB^{r'}_{r}\,Q^{z,z'}_{r}.
$$
\end{proposition}

\begin{proof}
We literally follow the argument in the proof of Proposition \ref{prop6.E}.
{}From the definition of the Meixner symmetric functions and \eqref{eq8.A} it
is readily seen that
$$
\LYB^{r'}_{r}\,\MM^{(z,z',r)}_\nu
=\left(\frac{r}{r'}\right)^{|\nu|}\,\MM^{(z,z',r')}_\nu
$$
and then \eqref{eq8.F} implies that the both sides of the operator equality in
question give the same result when applied to $\MM^{(z,z',r)}_\nu$. Therefore,
the equality holds on all elements of $\Sym$. As these elements separate points
of $\Y$, this concludes the proof.
\end{proof}

\begin{proposition}[cf. Proposition \ref{prop6.F}]\label{prop8.F}
The matrix  $\Qz$ satisfies the assumptions of Theorem \ref{thm2.1} with
functions $\ga(\la)=\eta(\la)=|\la|+1$, $\la\in\Y$.
\end{proposition}

\begin{proof}
As seen from the description of the $Q$-matrix given in \cite[Section
2.5]{BO-PTRF06} (see the sentence just before \cite[Proposition
2.11]{BO-PTRF06}), for any $\la\in\Y$ one has
$$
\sum_{\Box\in\la^\pm}Q^\z_r(\la,\la\pm\Box)=Q^{(c)}_r(|\la|, |\la|\pm1), \qquad
c:=zz'
$$
(note that $c>0$ because of Condition \ref{cond1.A}). This implies that the
action of $\Qz$ preserves the subspace in $\Fun(\Y)$ formed by those functions
in variable $\la\in\Y$ that depend only on $|\la|$, and in that subspace, the
action reduces to that of the difference operator $Q^{(c)}_r$ with $c=zz'$.

Therefore, the claim of the proposition reduces to that of Proposition
\ref{prop6.F}.
\end{proof}

Combining this proposition with Theorem \ref{thm2.1} we get

\begin{corollary}[cf. Corollary \ref{cor6.A}]\label{cor8.A}
{\rm(i)} The $Q$-matrix $\Qz$ gives rise to a Feller semigroup $T^\z_r(t)$ on
$C_0(\Y)$ whose generator $A^\z_r$ is implemented by $\Qz$.

{\rm(ii)} The subspace $C_0(\Y)$ is a core for generator $A^\z_r$.
\end{corollary}

\begin{proposition}[cf. Proposition \ref{prop6.K}]\label{prop8.K}
For every couple $(z,z')$ of parameters subject to Condition \ref{cond1.A}
there exists a unique Feller semigroup $T^\z(t)$ such that for every $r>0$,
$T^\z(t)$ is consistent with the $T^\z_r(t)$, $r>0$, in the sense that
\begin{equation*}
T^\z(t)\,\LYB^{\infty}_{r}=\LYB^{\infty}_{r}\,T^\z_r(t),\qquad t\ge 0, \quad
r>0.
\end{equation*}
\end{proposition}

\begin{proof}
The argument is exactly the same as in Proposition \ref{prop6.K}: We know that
the $Q$-matrices with various values of parameter $r$ are consistent with the
links (Proposition \ref{prop8.E}). It follows, by virtue of Proposition
\ref{prop5.C}, that the semigroups are consistent with the links, too.
Therefore, we may apply Proposition \ref{prop5.B}, which gives the desired
result.
\end{proof}

\begin{definition}\label{def8.A}
For $r>0$, we denote by $X^\z_r$ the Feller Markov process on $\Y$ determined
by the semigroup $T^\z_r(t)$. Likewise, we denote by $X^\z$ the Feller Markov
process on $\wt\Om$ determined by the semigroup $T^\z(t)$.
\end{definition}

\subsection{A family of cores for Markov semigroup generators}

The following two claims are used in the proposition below.

First, let $\om$ range over the Thoma cone $\wt\Om$ and $\mu$ range over $\Y$.
For any fixed $q\in(0,1)$, the functions $q^{|\om|} S_\mu(\om)$ span a dense
subspace in $C_0(\wt\Om)$, see \cite[Corollary 3.4.6]{BO-MMJ13}.

Second, let $\la$ range over $\Y$. Recall that in Section \ref{sect7.A} we
defined an embedding $\Y\hookrightarrow\wt\Om$ via the map $\la\mapsto
\om_\la$. Observe that $|\la|=|\om_\la|$; this implies that a sequence
$\{\la\}$ of diagrams goes to infinity in the discrete set $\Y$ if and only if
its image $\{\om_\la\}$ goes to infinity in the locally compact space $\wt\Om$.
Combining this with the first claim we conclude that for any fixed $q\in(0,1)$,
the functions $q^{|\la|} F_\mu(\la)$ span a dense subspace in $C_0(\Y)$.

\begin{proposition}[cf. Proposition \ref{prop6.G}]\label{prop8.G}
{\rm(i)} For any $r'>r>0$, the operator $\LYB^{r'}_{r}: C_0(\Y)\to C_0(\Y)$ has
a dense range.

{\rm(ii)} Likewise, for any $r>0$, the operator $\LYB^\infty_r: C_0(\Y)\to
C_0(\wt\Om)$ has a dense range.
\end{proposition}

\begin{proof}
(i) Fix an arbitrary $q\in(0,1)$ and let $\mu$ range over $\Y$. By
\eqref{eq8.C}, $\LYB^{r'}_{r}$ maps the linear span of functions $q^{|\la|}
S_\mu(\la)$ onto the linear span of functions $(q')^{|\la|} S_\mu(\la)$ with
some other $q'\in(0,1)$. Since these spans are dense, we get the desired claim.

(ii) The same argument, with reference to \eqref{eq8.I}.
\end{proof}

Denote by $\Az_r$ and $\Az$ the generators of the semigroups $\Tz_r(t)$ and
$\Tz(t)$, respectively.

\begin{proposition}[cf. Proposition \ref{prop6.H}]\label{prop8.H}
Fix an arbitrary number $q\in(0,1)$ and let $\mu$ range over $\Y$.

{\rm(i)} For every $r>0$, the linear span of functions $q^{|\la|} S_\mu(\la)$,
where argument $\la$ ranges over\/ $\Y$, is a core for $\Az_r$.

{\rm(ii)} Likewise, the linear span of functions $q^{|\om|} S_\mu(\om)$, where
argument $\om$ ranges over\/ $\wt\Om$, is a core for $\Az$.
\end{proposition}

\begin{proof}
(i) Observe that if $r_2>r_1>0$ and $\F_1$ is a core for $\Az_{r_1}$, then
$\F_2:=\LYB^{r_2}_{r_1}\F_1$ is a core for $\Az_{r_2}$. Indeed, by virtue of
claim (i) of Proposition \ref{prop8.G}, we may apply the argument of
Proposition \ref{prop5.A}.

Now take $r_2=r$ and $r_1=(1-q)r$. Then, as is seen from \eqref{eq8.B}, the
linear span of functions $q^{|\la|}S_\mu(\la)$ is just the image under
$\LYB^{r_2}_{r_1}$ of the space $C_c(\Y)$. By virtue of Proposition
\ref{prop8.F} and claim (iv) of Theorem \ref{thm2.1}, $C_c(\Y)$ is a core for
$\Az_{r_1}$. Therefore, its image is a core for $\Az_{r_2}$.

(ii) We argue as above. First, application of claim (ii) of Proposition
\ref{prop8.G} allows us to conclude that if $\F\subset C_0(\Y)$ is a core for
$\Az_r$ for some $r>0$, then $\LYB^\infty_r\F$ is a core for $\Az$.

Next, given $q\in(0,1)$ we take $r=-\log q$ and $\F=C_0(\Y)$. As pointed above,
$\F$ is a core for $\Az_r$. On the other hand, \eqref{eq8.H} shows that the
linear span of functions $q^{|\om|} S_\mu(\om)$ coincides with $\LYB^\infty_r
\F$.
\end{proof}

\subsection{The infinite-variate Laguerre differential operator}\label{sect8.A}
Following \cite[Theorem 4.10]{Ols-IMRN12}, we introduce the following partial
differential operator in countably many formal variables $e_1,e_2,\dots$:
\begin{equation}\label{eq8.J}
\begin{aligned}
\Dz&=\sum_{n\ge1}\left(\sum_{k=0}^{n-1}(2n-1-2k)e_{2n-1-k}e_k\right)
\frac{\pd^2}{\pd e_n^2}\\
&+2\sum_{n'>n\ge1}\left(\sum_{k=0}^{n-1}(n'+n-1-2k)e_{n'+n-1-k}e_k\right)
\frac{\pd^2}{\pd e_{n'}\pd e_n}\\
&+\sum_{n=1}^\infty\big(-ne_n+(z-n+1)(z'-n+1)e_{n-1}\big)\frac{\pd}{\pd e_n}
\end{aligned}
\end{equation}
with the agreement that $e_0=1$. We call it the {\it infinite-variate Laguerre
differential operator\/}.

Since all coefficients of $\Dz$ are given by finite sums, $\Dz$ is applicable
to any polynomial in $e_1,e_2,\dots$. This means that it is well defined on
$\Sym$ provided that we interpret our formal variables as the elementary
symmetric functions (here we use the fact that $\{e_1,e_2,\dots\}$ is a system
of algebraically independent generators of $\Sym$). But $\Dz$ is also
applicable to more general cylinder functions, in particular, to the functions
of the form $q^{e_1} F$, where $q\in(0,1)$ and $F\in\Sym$. Note that
$e_1(\om)=|\om|$, so that these are just the functions considered in claim (ii)
of Proposition \ref{prop8.H}. By virtue of this claim, for any fixed
$q=e^{-r}\in(0,1)$, the functions of the form $q^{e_1} F$ with $F\in\Sym$ enter
the domain of the generator $\Az$, and $\Az$ is uniquely determined by its
action on these functions.

\begin{proposition}[cf. Proposition \ref{prop6.I}]\label{prop8.I}
For any $r>0$, the action of the generator $\Az$ on the functions of the form
$\exp(-r e_1)F$ with $F$ ranging over the algebra $\Sym=\R[e_1,e_2,\dots]$ is
implemented by the infinite-variate Laguerre differential operator $\Dz$
defined by \eqref{eq8.J}.
\end{proposition}

\begin{proof}
Let $\mu$ range over $\Y$ and $m:=|\mu|$. Recall that $\mathbf1_\mu$ denotes
the delta function on $\Y$ concentrated at the point $\mu$. We also set
\begin{equation*}
\wt{\mathbf1}_\mu=(\dim\mu)^{-1}\mathbf1_\mu
\end{equation*}
and
$$
f_\mu=\frac{r^m}{m!}\,\exp(-re_1)\,S_\mu.
$$
By virtue of \eqref{eq8.H},
$$
\LYB^\infty_r\wt{\mathbf1}_\mu=f_\mu.
$$
Recall also that $\mu^+$ and $\mu^-$ denote the sets of boxes that can be
appended to or removed from $\mu$, respectively.

We are going to prove the following analogs of formulas \eqref{eq6.A} and
\eqref{eq6.B}:
\begin{multline}\label{eq8.M}
\Qc_r\wt{\mathbf1}_\mu=-[(2r+1)m+rzz']\wt{\mathbf1}_\mu
+(r+1)(m+1)\sum_{\Box\in\mu^+}\wt{\mathbf1}_{\mu+\Box}\\
+\frac{r}{m}\sum_{\Box\in\mu^-}(z+c(\Box))(z'+c(\Box))\wt{\mathbf1}_{\mu-\Box}
\end{multline}
and
\begin{multline}\label{eq8.N}
\Dz f_\mu=-[(2r+1)m+rzz']f_\mu
+(r+1)(m+1)\sum_{\Box\in\mu^+}f_{\mu+\Box}\\
+\frac{r}{m}\sum_{\Box\in\mu^-}(z+c(\Box))(z'+c(\Box))f_{\mu-\Box}
\end{multline}
(for the empty diagram $\mu$, the set $\mu^-$ is empty and the corresponding
sum disappears). These formulas show that the operator $\Qz$ acts on the
functions $\wt{\mathbf1}_\mu$ in exactly the same way as the operator $\Dz$
acts on the functions $f_\mu$, which implies the claim of the proposition.

The proof of \eqref{eq8.M} is trivial: this formula directly follows from the
very definition of $\Qz$, see \eqref{eq8.L}.

The proof of \eqref{eq8.N} is a bit more complicated. Observe that if a second
order partial differential operator $D$, symbolically written as
$$
D=\sum_{i,i}c_{ij}\pd_i\pd_j+\textrm{first order terms},
$$
is applied to a product of two functions, $GF$, then the result can be written
as the sum of three expressions:
\begin{equation}\label{eq8.W}
D(GF)=\underbrace{(DF)G}_{1}+\underbrace{G(DF)}_{2}+\underbrace{\sum_{i,j}c_{ij}[(\pd_i
G)(\pd_j F)+(\pd_j G)(\pd_i F)]}_{3}.
\end{equation}

Let us apply this general formula to
$$
D:=\Dz, \quad G:=\frac{r^m}{m!}e^{-re_1}, \quad F:=S_\mu
$$
and examine the corresponding three expressions arising from \eqref{eq8.W}.

1. The first expression is equal to
\begin{gather*}
\frac{r^m}{m!}\left(\Dz e^{-re_1}\right)S_\mu
=\frac{r^m}{m!}\left\{\left[e_1\frac{d^2}{de_1^2}
+(-e_1+zz')\frac{d}{de_1}\right]e^{-re_1}\right\}S_\mu\\
=\frac{r^m}{m!}(r^2+r)e^{-re_1}e_1S_\mu-\frac{r^m}{m!}rzz'e^{-re_1}S_\mu.
\end{gather*}
It is well known that
$$
e_1S_\mu=\sum_{\Box\in\mu^+}S_{\mu+\Box}.
$$
It follows that the first expression in question is equal to
\begin{multline}\label{eq8.W1}
(r+1)(m+1)\frac{r^{m+1}}{(m+1)!}e^{-re_1}\sum_{\Box\in\mu^+}S_{\mu+\Box}
-rzz'\frac{r^m}{m!}e^{-re_1}S_\mu\\
=(r+1)(m+1)\sum_{\Box\in\mu^+}f_{\mu+\Box}-rzz'f_\mu.
\end{multline}

2. The second expression in \eqref{eq8.W} takes the form
$$
\frac{r^m}{m!}e^{-re_1}(\Dz S_\mu).
$$
It follows from \cite[Theorem 4.1 and Definition 4.7]{Ols-IMRN12} that
$$
\Dz S_\mu=-m S_\mu+\sum_{\Box\in\mu^-}(z+c(\Box))(z'+c(\Box))S_{\mu-\Box},
\qquad m=|\mu|.
$$
This implies that the second expression is equal to
\begin{equation}\label{eq8.W2}
-mf_\mu+\frac{r}{m}\sum_{\Box\in\mu^-}(z+c(\Box))(z'+c(\Box))f_{\mu-\Box}.
\end{equation}

3. The only relevant part of our differential operator $D=\Dz$ that contributes
to the third expression in \eqref{eq8.W} is
$$
e_1\frac{\pd^2}{\pd e_1^2}+2\sum_{n'>1}n'e_{n'}\frac{\pd^2}{\pd e_{n'}\pd e_1},
$$
because the remaining terms in $\Dz$ are either of the first order or do not
contain the partial derivative in variable $e_1$ while our function $G$ depends
on $e_1$ only. It follows that the third expression has the form
\begin{multline*}
2\frac{r^m}{m!}\left(\frac{d}{de_1}e^{-re_1}\right)e_1\frac{\pd}{\pd e_1}S_\mu
+2\frac{r^m}{m!}\left(\frac{d}{de_1}e^{-re_1}\right)\sum_{n'>1}n'e_{n'}\frac{\pd}{\pd
e_{n'}}S_\mu\\
=-2r\frac{r^m}{m!}e^{-re_1}\sum_{n\ge1}ne_n\frac{\pd}{\pd e_n}S_\mu.
\end{multline*}

Observe that the operator
$$
\sum_{n\ge1}ne_n\frac{\pd }{\pd e_n}
$$
is the ``Euler operator''; its action on the homogeneous function $S_\mu$
amounts to multiplication by its degree $m$. Using this fact we see that the
third expression is equal to
\begin{equation}\label{eq8.W3}
-2rmf_\mu.
\end{equation}

Finally, summing up \eqref{eq8.W1}, \eqref{eq8.W2}, and \eqref{eq8.W3} we get
the desired formula \eqref{eq8.N}

\end{proof}

The {\it Laguerre symmetric functions\/}, introduced in Olshanski
\cite{Ols-IMRN12}, are elements of $\Sym$ depending on parameters $z$ and $z'$,
and indexed by Young diagrams $\nu\in\Y$:
\begin{equation}\label{eq8.O}
\LL^\z_\nu=\sum_{\mu:\,\mu\subseteq\nu} (-1)^{|\nu|-|\mu|}
\frac{\dim\nu/\mu}{(|\nu|-|\mu|)!}\, (z)_{\nu/\mu}(z')_{\nu/\mu} S_\mu.
\end{equation}
As shown in \cite{Ols-IMRN12}, they form a basis in $\Sym$ diagonalizing
operator $\Dz$:
\begin{equation}\label{eq8.P}
\Dz\LL^\z_\nu=-|\nu|\LL^\z_\nu, \qquad \nu\in\Y.
\end{equation}

The above formula is similar to \eqref{eq6.7}, and the next formula is an
analog of \eqref{eq6.14}:
\begin{equation}\label{eq8.Q}
\LYB^\infty_r \, \MM^{(z,z',r)}_\nu=r^{|\nu|}\LL^\z_\nu, \qquad r>0, \quad
\nu\in\Y.
\end{equation}
The proof of \eqref{eq8.Q} is easy and analogous to that of \eqref{eq6.14}.
Namely, we compare the expansions of the Meixner and Laguerre symmetric
functions in the bases $\{\FS_\mu\}$ and $\{S_\mu\}$, respectively (see
\eqref{eq8.D} and \eqref{eq8.O}), and then apply \eqref{eq8.G}, which says that
$\LYB^\infty_r$ takes $\FS_\mu$ to $r^{|\mu|}S_\mu$.

\subsection{Approximation}
Recall that in Section \ref{sect7.A} we introduced an embedding
$\la\mapsto\om_\la$ of the set $\Y$ into the Thoma cone $\wt\Om$. Now let us
introduce a family of embeddings $\varphi_r:\Y\hookrightarrow\wt\Om$ depending
on parameter $r>0$:
$$
\varphi_r(\la)=r^{-1}\om_\la, \qquad \la\in\Y,
$$
where multiplication by constant factor $r^{-1}$ in the right-hand side means
that all coordinates of $\om_\la$ are multiplied by that constant --- a natural
homothety on the cone. Obviously, $\varphi_1$ is the map $\la\mapsto\om_\la$.

The latter map should be viewed as a counterpart of the inclusion map
$\Z_+\hookrightarrow\R_+$, while $\varphi_r$ is a counterpart of the scaled
embedding $\Z_+\ni l\mapsto r^{-1}l\in\R_+$.

Note that $\varphi_r(\Y)$ is a discrete, locally finite subset of $\wt\Om$.
Therefore, we may define the projection $\pi_r: C_0(\wt\Om)\to C_0(\Y)$ as in
Section \ref{sect5.A}:
$$
(\pi_rf)(\la)=f(\varphi_r(\la)), \qquad \la\in\Y.
$$
It is used in the proposition below to define the approximation procedure.

Recall that in Corollary \ref{cor8.A} and Proposition \ref{prop8.K} we defined
Feller semigroups $T^\z_r(t)$ and $T^\z(t)$ acting on the Banach spaces
$C_0(\Y)$ and $C_0(\wt\Om)$, respectively.

\begin{proposition}[cf. Proposition \ref{prop6.J}]\label{prop8.J}
Let $(z,z')$ be fixed. As $r\to+\infty$, the semigroups $T^\z_r(t)$ approximate
the semigroup $T^\z(t)$ in the sense of Definition \ref{def3.B}.
\end{proposition}

\begin{proof}
As in the proof of Proposition \ref{prop6.J}, we only need to check all the
hypotheses of Proposition \ref{prop5.D}.

Again, the assumptions stated in Section \ref{sect5.B} are satisfied: we know
that $\YB$ is a Feller system, the semigroups $T^\z_r(t)$ with varying
parameter $r>0$  are consistent with the links of $\YB$, they are  Feller
links, and, by the very definition, the semigroup $T^\z(t)$ is the boundary
semigroup determined by the pre-limit semigroups $T^\z_r(t)$.

Next, we have to check Conditions \ref{cond4.A} and \ref{cond5.A}.

The second condition consists in the requirement that $C_c(\Y)$ is a core for
generator $A^\z_r$ and, moreover, is invariant under its action. The fact that
$C_c(\Y)$ is a core follows from Corollary \ref{cor8.A}, item (ii). Its
invariance follows from item (i), because $C_c(\Y)$ is obviously invariant
under the action of $\Qz$.

Finally, the first condition means that that for fixed $s>0$ and $\mu\in\Y$
\begin{equation}\label{eq8.U}
\lim_{r\to \infty}\sup_{\la\in
\Y}\left|\LYB^{r}_{s}(\la,\mu)-\LYB^\infty_{s}(\varphi_r(\la),\mu)\right|=0,
\end{equation}
and this was established in the proof of \cite[Theorem 3.4.7]{BO-MMJ13} (note
only a slight divergence of notation: in \cite{BO-MMJ13}, we wrote $r'$ and $r$
instead of $r$ and $s$, respectively).
\end{proof}

\subsection{The stationary distribution}

The so-called {\it mixed z-measure\/} on $\Y$ with parameters $(z,z')$ and $r$
is defined by
\begin{equation}\label{eq8.V}
M^\z_r(\la)=(r+1)^{-zz'}\left(\frac{r}{r+1}\right)^{|\la|}
\cdot\prod_{\Box\in\la}(z+c(\Box))(z'+c(\Box))\cdot
\left(\frac{\dim\la}{|\la|!}\right)^2,
\end{equation}
where $\la$ ranges over $\Y$. As before, we assume that $r>0$ and $(z,z')$
satisfies Condition \ref{cond1.A}. Then the weights $M^\z_r(\la)$ are strictly
positive and sum to 1, so that $M^\z_r$ is a probability measure on $\Y$ whose
support is the whole set $\Y$. Measures $M^\z_r$ first appeared in
Borodin--Olshanski \cite{BO-CMP00}; additional information can be found in
Okounkov \cite{Ok-SL2-MSRI01}, Borodin--Olshanski \cite{BO-Gamma-Adv05}, and
Olshanski \cite{Ols-IMRN12}. These measures are a particular case of Okounkov's
{\it Schur measures\/} introduced in \cite{Ok-Selecta01}. (As mentioned above,
in those papers, the third parameter, denoted by $\xi$, is related to our
parameter $r$ by $\xi=r(1+r)^{-1}$.)

\begin{proposition}\label{prop8.L}
$M^\z_r$ serves as a unique stationary distribution for the Markov process
$X^\z_r$ determined by the Feller semigroup $T^\z_r(t)$.
\end{proposition}

\begin{proof}
The fact that $M^\z_r$ is a stationary measure is a particular case of
\cite[Proposition 2.12]{BO-PTRF06}.

Next, from the structure of matrix $Q^\z_r$ and the construction of $X^\z_r$ it
follows that $X^\z_r$ is an irreducible Markov chain: all states $\la\in\Y$ are
communicating. According to a general theorem (see Anderson \cite[Chapter 5,
Theorem 1.6]{And91}) this implies the uniqueness claim.
\end{proof}

As shown in \cite[Proposition 3.5.3]{BO-MMJ13}, the measures $M^\z_r$ with
varying parameter $r$ are compatible with the links $\LYB^{r'}_r$, that is,
they form a coherent family. Therefore, they give rise to a boundary measure on
$\wt\Om$, which we denote by $M^\z$ and call the {\it z-measure on the Thoma
cone}.

\begin{proposition}\label{prop8.M}
$M^\z$ serves as a unique stationary distribution for the Markov process $X^\z$
determined by the Feller semigroup $T^\z(t)$.
\end{proposition}

\begin{proof}
We have to prove that $M^\z$ satisfies the relation $M^\z T^\z(t)=M^\z$ and is
a unique probability measure on $\wt\Om$ with this property. By Proposition
\ref{prop8.L}, a similar claim holds for measures $M^\z_r$. Because
$\{M^\z_r:r>0\}$ is a coherent family, this immediately implies the desired
claim: an easy formal argument can be found in \cite[Section 2.8]{BO-JFA12}.
\end{proof}

\begin{proposition}\label{prop8.N}
$M^\z$ is the weak limit of measures $\varphi_r(M^\z_r)$ as $r\to+\infty$.
\end{proposition}

\begin{proof}
Indeed, as mentioned above (see the proof of Proposition \ref{prop8.J}), in our
situation Condition \ref{cond4.A} is satisfied. Therefore, we may apply
Proposition \ref{prop4.A} which gives the desired result.
\end{proof}

\begin{remark}
In Olshanski \cite{Ols-IMRN12}, the z-measures on the Thoma cone were defined
in a different way, see \cite[Theorem 5.18]{Ols-IMRN12}. However, the two
definitions are equivalent, as can be seen from the comparison of Proposition
\ref{prop8.N} with \cite[Theorem 5.28]{Ols-IMRN12}.

\end{remark}

\section{Determinantal structure}\label{sect9}

\subsection{Generalities on correlation functions}

Let $\X$ be a locally compact metrizable separable space (we will actually take
for $\X$ the punctured real line $\R^*:=\R\setminus\{0\}$ or the
one-dimensional lattice). A finite or countably infinite collection of points
in $\mathfrak X$ without accumulation points is called a {\it configuration\/}.
We say ``collection'' and not ``subset'' because, in principle, multiple points
are permitted; one could also use the term ``multiset''. To a configuration
$\om$ we assign the Radon measure
$$
\De(\om):=\sum_{x\in\om}\De_x,
$$
where $\De_x$ denotes the delta-measure at $x$. This assignment establishes a
one-to-one correspondence between all possible configurations in $\mathfrak X$
and all sigma-finite Radon measures on $\X$ with the property that the mass of
any compact subset is a nonnegative integer. The space of configurations will
be denoted by $\Conf(\X)$. We equip it with the topology inherited from the
vague topology on the space of Radon measures. In particular, $\Conf(\X)$ has a
natural Borel structure. This structure is generated by the integer-valued
functions $\mathcal N_B$, where $B\subset\X$ is an arbitrary relatively compact
Borel subset and
$$
\mathcal N_B(\om):=|\om\cap B|, \qquad \om\in\Conf(\X).
$$

Let $M$ be a probability Borel measure on $\Conf(\X)$. Then the functions
$\mathcal N_B$ become random variables. We will assume that every such function
has finite moments of any order,
$$
\E_M((\mathcal N_B)^k)<+\infty, \qquad \forall k=1,2,\dots, \quad \forall B,
$$
where $\E_M$ means expectation relative to $M$. Under this assumption one
assigns to $M$ an infinite collection $\{\rho_k: k=1,2,\dots\}$ of measures,
where $\rho_k$ is a (usually infinite) measure on the $k$-fold product space
$\X^k$, defined as follows.

First, given $\om\in\Conf(\X)$, we form a purely atomic measure $\De^k(\om)$ on
$\X^k$ by setting
$$
\De^k(\om):=\sum_{x_1,\dots,x_k}\De_{x_1}\otimes\dots\otimes\De_{x_k},
$$
where the sum is taken on arbitrary ordered $k$-tuples of distinct points
extracted from $\om$.

Second, we interpret $\De^k(\om)$ as a random measure driven by the probability
distribution $M$ and average over $M$,
$$
\rho_k=\rho_k^M:=\E_M(\De^k(\,\cdot\,)).
$$

The measure $\rho_k^M$ is called the $k$th {\it correlation measure\/} of $M$,
and the first correlation measure $\rho_1^M$ is also called the {\it density
measure\/}. Under mild hypotheses on the correlation measures, they determine
the initial measure $M$ uniquely, see Lenard \cite{Len-CMP73}.

$M$ is said to be a {\it determinantal measure\/} if the following condition
holds. Choose a ``reference'' measure $\si$ on $\X$, equivalent to the density
measure (the condition stated below does not depend on the choice of $\si$).
There should exist a function $K(x,y)$ on $\X\times\X$ such that, for every
$k\ge1$, the $k$th correlation measure $\rho^M_k$ is absolutely continuous with
respect to $\si^{\otimes k}$, and the corresponding Radon-Nikod\'ym density is
given by a $k\times k$ principal minor extracted from kernel $K$:
$$
\frac{\rho^M_k}{\si^{\otimes k}}(x_1,\dots,x_k)=\det[K(x_i,x_j)].
$$

The quantity in the left-hand side is called the $k$th {\it correlation
function\/}, and $K(x,y)$ is called the {\it correlation kernel\/} of $M$. In
contrast to correlation functions, the correlation kernel, if it exists, is not
a canonical object: there are ways to modify it without affecting the
correlation functions. On the other hand, any determinantal measure is uniquely
determined by its correlation functions and hence by the correlation kernel.

``Determinantal measure'' is another name for ``determinantal point process''
(more precisely, for the {\it law\/} of such a point process). A standard
reference is Soshnikov's expository paper \cite{Sosh-RussMathSurv00}. See also
the more recent survey Borodin \cite{Bor-Oxford11} and references therein.

\subsection{Determinantal structure of the stationary distributions}

Set $\X=\R^*$ and define a map $\wt\Om\to\Conf(\R^*)$ as follows:
$$
\wt\Om\ni\om\mapsto \bar\om:=\{\al_i: \al_i\ne0\}\cup\{-\be_i:
\be_i\ne0\}\in\Conf(\R^*).
$$
Because of the constraint $\sum\al_i+\sum\be_i\le\de<+\infty$, $\bar\om$ is
indeed a configuration on $\R^*$. Clearly, the map is continuous and hence
Borel. So it converts every probability Borel measure $M$ on $\wt\Om$ to a
probability Borel measure $\bar M$ on $\Conf(\R^*)$. This makes it possible to
speak about the correlation functions of $M$, referring to those of $\bar M$.

We fix a pair of parameters $(z,z')$ satisfying Condition \ref{cond1.A} and
denote by $\bar M^\z$ the measure on $\Conf(\R^*)$ coming from the z-measure
$M^\z$. In the next theorem, $K^\z(x,y)$ denotes the Whittaker kernel on
$\R^*\times\R^*$ studied in Borodin \cite{Bor-AA00} and Borodin--Olshanski
\cite{BO-CMP00}, \cite{BO-PTRF06}. We will not use its exact form here.

\begin{theorem}\label{thm9.A}
$\bar M^\z$ is a determinantal measure whose correlation kernel is the
Whittaker kernel $K^\z(x,y)$.
\end{theorem}

This result was first proved in \cite{Bor-AA00}. Below we give in detail a
different derivation, because it is well suited for the extension to the case
of finite-dimensional distributions of processes $X^\z$. Note that a similar
argument is contained in \cite[Proposition 4.2]{BO-Beta-EJComb05}.

\begin{proof}
Step 1. Let $M$ be a probability measure on $\wt\Om$ and $\bar M$ be the
corresponding measure on $\Conf(\R^*)$. We will establish a simple estimate
which, in particular, provides a convenient sufficient condition for the
existence of the correlation measures.

For $\epsi>0$, set
$$
B_\epsi:=\R\setminus(-\epsi,\epsi)\subset\R^*.
$$
Recall the notation $|\om|=|(\al,\be,\de)|=\de$. The basic constraint
$\sum(\al_i+\be_i)\le|\om|$ implies the inequality
\begin{equation}\label{eq9.C}
|\bar\om\cap B_\epsi|\le \epsi^{-1}|\om|,
\end{equation}
which in turn implies that
$$
\E_{\bar M}((\mathcal N_{B_\epsi})^k)\le\epsi^{-k}\int_{\wt\Om}|\om|^kM(d\om),
\qquad k=1,2,\dots\,.
$$
Denote by $|M|$ the measure on $\R_+$ that is the pushfoward of $M$ under the
projection $\om\mapsto|\om|$. The above inequality can be rewritten as
$$
\E_{\bar M}((\mathcal N_{B_\epsi})^k)\le\epsi^{-k}\int_{\R_+}s^k |M|(ds) \qquad
k=1,2,\dots\,.
$$
This shows that if $|M|$ has finite moments of all orders, then the left-hand
side is finite for all $k$ and hence the correlation measures of $\bar M$ are
well defined. (Here we tacitly used the evident fact that any compact subset of
$\R^*$ is contained in subset $B_\epsi$ with $\epsi$ small enough.)

Step 2. For $r>0$, set $M_r:=M\La^\infty_r$. It is initially defined as a
probability distribution on $\Y$, but it is convenient to transfer it to
$\wt\Om$ using the embedding $\varphi_r:\Y\to\wt\Om$. So, we will regard each
$M_r$ as a probability distribution on $\wt\Om$.

By Proposition \ref{prop4.A}, $M_r$ converges to $M$ in the weak topology as
$r\to+\infty$, meaning that
\begin{equation}\label{eq9.B}
\lim_{r\to+\infty}\langle\Psi,M_r\rangle=\langle\Psi,M\rangle
\end{equation}
for any continuous bounded function $\Psi$ on $\wt\Om$.

Assume now that we dispose of the following uniform bound on the tails of
measures $|M_r|$:
\begin{equation}\label{eq9.A}
\begin{gathered}
\text{For every $k=1,2,\dots$, one has $\int_{\R_+}s^k|M_r|(ds)\le C_k$}\\
\text{with a constant $C_k$ independent on $r$.}
\end{gathered}
\end{equation}
Then, evidently, \eqref{eq9.B} holds under weaker assumptions on $\Psi$: it
suffices to require that $\Psi$ is continuous and has {\it moderate growth at
infinity\/}, meaning that $|\Psi(\om)|\le\const(1+|\om|)^k$ for some $k$.

Step 3. Assume that condition \eqref{eq9.A} is satisfied. We claim that then
the correlation measures of $M_r$ vaguely converge to the respective
correlation measures of $M$.

Indeed, first of all, by virtue of step 1, our assumption guarantees the very
existence of the correlation measures for measures $M_r$. Moreover, the
inequalities \eqref{eq9.A} are inherited by the limit measure $M$, so that its
correlation measures exist, too.

Fix $k=1,2,\dots$\,. By definition, the vague convergence of the $k$th
correlation measures, $\rho_k^{M_r}\to\rho_k^M$, means that
$$
\lim_{r\to+\infty}\langle F,\rho_k^{M_r}\rangle=\langle F,\rho_k^M\rangle
$$
for any continuous, compactly supported function $F$ on $(\R^*)^k$. By the very
definition of the correlation measures, the latter relation is equivalent to
fulfillment of relation \eqref{eq9.B}, where $\Psi=\Psi_F$ has the following
form
\begin{equation}\label{eq9.D}
\Psi_F(\om)=\langle F,
\De^k(\bar\om)\rangle=\sum_{x_1,\dots,x_k}F(x_1,\dots,x_k),
\end{equation}
where the sum is taken over ordered $k$-tuples of distinct points extracted
from configuration $\bar\om$.

Now, by virtue of step 2, it suffices to check that $\Psi_F$ is continuous and
has moderate growth at infinity.

Choose $\epsi$ so small that the support of $F$ is contained in $B_\epsi^k$. By
virtue of bound \eqref{eq9.C},
$$
|\Psi_F(\om)|\le \epsi^{-k}\Vert F\Vert \,|\om|^k.
$$
Therefore, $\Psi_F$ has moderate growth at infinity.

To see that $\Psi_F$ is continuous look at the right-hand side of \eqref{eq9.D}
and observe that $F(x_1,\dots,x_k)$ vanishes unless all quantities
$|x_1|,\dots,|x_k|$ are bounded from below by $\epsi$, which in turn entails
that the $k$-tuple $\{x_1,\dots,x_k\}$ is contained in the subset
$$
\{\al_1,\dots,\al_m,-\be_1,\dots,-\be_m\}, \qquad m:=[\epsi^{-1}|\om|].
$$
That is, only coordinates of $\om$ with a few first indices really contribute,
and this finite set of possible indices depends only on $|\om|$. Together with
the continuity of $F$ this gives the desired claim.

Step 4. Now we apply the above general arguments to $M:=M^\z$ and the
corresponding pre-limit measures $M_r:=M^\z_r$. Recall that, according to our
convention, $M^\z_r$ lives on $\varphi_r(\Y_r)\subset\wt\Om$. Then we know
exactly what is $|M^\z_r|$: it is a scaled negative binomial distribution
living on the subset $r^{-1}\Z_+\subset\R_+$:
$$
|M^\z_r|(r^{-1}l)=(r+1)^{-zz'}\frac{(zz')^{\down l}}{l!}\left(\frac
r{r+1}\right)^l, \qquad l\in\Z_+.
$$
Condition \eqref{eq9.A} on the tails is readily checked (note that the limiting
measure $|M^\z|$ is the $\Ga$-distribution with parameter $zz'$). Therefore,
all correlation functions exist, and we have the limit relation
$$
\lim_{r\to+\infty}\langle F,\rho_k^{M_r}\rangle=\langle F,\rho_k^M\rangle,
\qquad M_r:=M^\z_r, \quad M:=M^\z
$$
for any continuous compactly supported function $F$ on $(\R^*)^k$.

Step 5. Finally, we apply the results of our papers \cite{BO-CMP00} and
\cite{BO-PTRF06}. As shown in those papers, the pre-limit measures $M_r=M^\z_r$
are determinantal, with some correlation kernels $K^\z_r(x,y)$, called {\it
discrete hypergeometric kernels\/}, for which an explicit expression is known.

In accordance with our definition of measure $M^\z_r$, it lives on the lattice
$r^{-1}\Z'\subset\R^*$, where $\Z':=\Z+\frac12$. As the reference measure
$\si$, we take  the counting measure on the lattice. Then one can write
$$
\langle F,\rho_k^{M_r}\rangle=\sum_{(x_1,\dots,x_k)\in
(r^{-1}\Z')^k}F(x_1,\dots,x_k)\det[K^\z_r(x_i,x_j)].
$$

On the other hand, the limiting behavior of kernels $K^\z_r(x,y)$ was studied
in \cite[Theorem 5.4]{BO-CMP00}. It follows that, as $r\to+\infty$, the
right-hand side of the above relation converges to
$$
\int_{(x_1,\dots,x_k)\in(\R^*)^k} F(x_1,\dots,x_k)\det[K^\z(x_i,x_j)]dx_1\dots
dx_k,
$$
where $K^\z(x,y)$ is the Whittaker kernel. This completes the proof.

\end{proof}

\begin{remark}\label{rem9.A}
The map $M\mapsto\bar M$ converting a measure on $\wt\Om$ to that on
$\Conf(\R^*)$ is not injective, because the map $\om\mapsto\bar\om$ ignores
parameter $\de$. However, $M$ is uniquely determined by its pushforward $\bar
M$ if it is known a priori that $M$ is supported by the subset
$$
\wt\Om_0:=\{\om: \sum\al_i+\sum\be_i=\de\}\subset\wt\Om.
$$
(Note that $\wt\Om_0$ is a dense Borel subset of type $G_\de$.)

This is just the case for $M=M^\z$, as can be proved using Olshanski
\cite[Theorem 6.1]{Ols-Birkh03}.  Therefore, $M^\z$ is completely specified by
the correlation kernel $K^\z(x,y)$ of the measure $\bar M^\z$.
\end{remark}

\subsection{Determinantal structure of equilibrium finite-dimensional distributions}

Starting Markov process $X^\z$ at time $t=0$ from the stationary distribution
we get a stationary in time stochastic process $\wt X^\z$. Given time moments
$0\le t_1<\dots<t_n$, let $M^\z(t_1,\dots,t_n)$ stand for the corresponding
finite-dimensional distribution of $\wt X^\z$. The distributions
$M^\z(t_1,\dots,t_n)$ are invariant under simultaneous shift of all time
moments by a constant; they can be called the {\it equilibrium\/}
finite-dimensional distributions. For $n=1$, we have $M^\z(t)\equiv M^\z$.

Initially $M^\z(t_1,\dots,t_n)$ is defined as a probability measure on the
$n$-fold product space $\wt\Om^n$, but then we convert it to a probability
measure $\bar M^\z(t_1,\dots,t_n)$ on $(\Conf(\R^*))^n$, just as we did above
for the case $n=1$. Observe that $(\Conf(\R^*))^n$ can be identified, in a
natural way, with $\Conf(\,\underbrace{\,\R^*\sqcup\dots\sqcup\R^*}_n\,)$. This
shows that we can interpret $\bar M^\z(t_1,\dots,t_n)$ as a probability
distribution on configurations, and the next theorem says that it is again in
the determinantal class.  This means that the correlation functions of $\bar
M^\z(t_1,\dots,t_n)$ are described by a ``dynamical'' (or ``space-time'')
kernel $K^\z(x,s;y,t)$ on $(\R^*\times\R)\times(\R^*\times\R)$ whose two
arguments, couples $(x,s)$ and $(y,t)$,  should be viewed as space-time
variables ranging over space-time $\R^*\times\R$. Given an arbitrary finite
collection $(x_1,t_1),\dots,(x_k,t_k)$, the $k\times k$ determinant
$$
\det\left[K^\z(x_i,t_i; x_j,t_j)\right]
$$
multiplied by $dx_1\dots dx_k$ gives the probability of the event that at each
prescribed moment $t_i$ (where $i=1,\dots,k$), the configuration
$\bar\om\in\Conf(\R^*)$ corresponding to $\om:=X^\z(t_i)$ contains a point in
the infinitesimal neighborhood $dx_i$ about position $x_i$, for every
$i=1,\dots,k$.

The kernel $K^\z(x,s;y,t)$ in question is the {\it extended Whittaker
kernel\/}; we refer to Borodin--Olshanski \cite{BO-PTRF06} for its description.

\begin{theorem}\label{thm9.B}
The pushforwards $\bar M^\z(t_1,\dots,t_n)$ of the equilibrium
finite-dimen\-sional distributions $M^\z(t_1,\dots,t_n)$ are determinantal
measures described by the extended Whittaker kernel $K^\z(x,s;y,t)$.
\end{theorem}

This is a generalization of Theorem \ref{thm9.A}, which is a particular case of
Theorem \ref{thm9.B} for $n=1$, because $M^\z(t)\equiv M^\z$, and
$K^\z(x,s;y,t)$ reduces to the Whittaker kernel $K^\z(x,y)$ for $s=t$.

\begin{remark}[cf. Remark \ref{rem9.A}]\label{rem9.B}
Note that measure $M^\z(t_1,\dots,t_n)$ is supported by the subset
$\wt\Om_0^n$, because every its one-dimensional marginal coincides with $M^\z$
and the latter measure is supported by $\wt\Om_0$. As in the case $n=1$, this
implies that the equilibrium finite-dimensional distributions are uniquely
determined by the extended Whittaker kernel.
\end{remark}

\begin{proof}[Proof of Theorem \ref{thm9.B}]
The argument for Theorem \ref{thm9.A} extends smoothly, with a few minor
evident modifications only. Let $M^\z_r(t_1,\dots,t_n)$ stand for the pre-limit
equilibrium finite-dimensional distributions. Corollary \ref{cor3.1} tells us
that they approximate the distributions $M^\z(t_1,\dots,t_n)$. To bound the
tails we use the fact, mentioned above, that the one-dimensional marginals
coincide with the stationary distribution. The correlation functions of the
pre-limit distributions are described by the extended version of the discrete
hypergeometric kernel, which converges to the extended Whittaker kernel as
$r\to+\infty$: this is established in \cite{BO-PTRF06}.
\end{proof}

\section{Remarks on the Plancherel limit}\label{sect10}

Let us return to the context of Section \ref{sect8.B}. So far the basic
parameters $z$ and $z'$ were fixed, but here we take a limit transition in
formulas \eqref{eq8.L} assuming that $z$ and $z'$ go to infinity while the
third parameter $r$ goes to 0 in such a way that the product $rzz'$ tends to a
fixed real number $\theta>0$. One may simply assume that $r$ is related to the
couple $\z$ by $r=\theta(zz')^{-1}$; recall that because of Condition
\ref{cond1.A}, $zz'$ is strictly positive, so that the above relation is
compatible with the fact that $r$ should be a positive number. The quantity
$\theta$ becomes our new parameter.

It is not difficult to verify that in this limit transition, all results of
Section \ref{sect8.B} survive. Namely, the $Q$-matrix $\Qz$ turns into the
following matrix:
\begin{equation}\label{10.A}
\begin{aligned}
Q_\theta(\la,\la+\square)&=\theta \frac{\dim(\la+\square)}{(|\la|+1)\dim\la},
\quad
\square\in\la^+,\\
Q_\theta(\la,\la-\square)&=\frac{|\la|\dim(\la-\square)}{\dim\la}, \quad \square\in\la^-,\\
-Q_\theta(\la, \la)&=|\la| +\theta.
\end{aligned}
\end{equation}
An analog of Proposition \ref{prop8.D} holds, with the Meixner symmetric
functions being replaced by the so-called {\it Charlier symmetric functions\/},
introduced in \cite{Ols-IMRN12} (these are obtained from the Meixner functions
via the same limit transition). A key observation is that the links
$\LYB^{r'}_r$ depend on parameters $r$ and $r'$ through their ratio $r/r'$,
which remains intact under the limit (it translates into the ratio
$\theta/\theta'$). Because of this fact, all other results of Section
\ref{sect8.B} are smoothly extended, too. We only have to change the notation
$r\to\theta$. Finally, we get a family $\{X_\theta: \theta>0\}$ of continuous
time Feller Markov chains on $\Y$.

Further, one can prove that $X_\theta$ has a unique stationary distribution,
which is nothing else than the the well-known {\it Poissonized Plancherel
measure\/}, first introduced in Baik--Deift--Johansson \cite{BDJ99}:
$$
M_\theta(\la)=e^{-\theta}\theta^{|\la|}\left(\frac{\dim\la}{|\la|!}\right)^2.
$$
It is a degeneration of the mixed z-measure \eqref{eq8.V}, which played an
important role in Borodin--Okounkov--Olshanski \cite{BOO-JAMS}.

The Markov chains $X_\theta$ were studied in our paper \cite{BO-Planch}. As
shown in that paper, $X_\theta$ admits a nice description in terms of the
Poisson process in the quarter-plane and the Robinson--Schensted algorithm.

The formalism of the present paper says that the family $\{X_\theta:
\theta>0\}$ gives rise to a Feller Markov process $X$ on the boundary $\wt\Om$,
and $X$ has a unique stationary distribution $M:=\varprojlim M_\theta$, the
boundary measure corresponding to the family of the Poissonized Plancherel
measures. On the other hand, it is readily seen that this boundary measure is
simply the Dirac measure at the point
$$
\om_1:=(\al=\underline0,\; \beta=\underline0, \;\de=1)\in\wt\Om,
$$
where $\underline0:=(0,0,\dots)$ is the null sequence.

At first glance, this looks strange, but the key is that $X$ is not a genuine
Markov process, but a {\it deterministic process\/}. Its transition function
$P(t)$ degenerates to a semigroup of continuous maps $\wt\Om\to\wt\Om$ which
have the following form:
$$
P(t): (\al,\be,\de)\mapsto (e^{-t}\al,\; e^{-t}\be,\; e^{-t}\de+(1-e^{-t})),
\qquad t\ge0.
$$
{}From this formula it is seen that, as $t\to+\infty$, $P(t)$ contracts the
whole space $\wt\Om$ to the point $\om_1$. There is no contradiction, because
such a deterministic process is formally a Markov process.

On the algebraic level, this phenomenon is clearly seen when we compute the
generator of $X$ as an operator in the algebra of symmetric functions: In
contrast to the Laguerre operator \eqref{eq8.J} we get a {\it first order\/}
differential operator. This operator is best written in terms of the generators
$p_1,p_2,\dots$ (the power-sum symmetric functions, see Section \ref{sect7.A}),
it has the form
$$
(1-p_1)\frac{\pd}{\pd p_1}+\sum_{n\ge2}np_n\frac{\pd}{\pd p_n}\,.
$$

The above discussion shows that our abstract formalism of constructing boundary
Markov processes via Markov intertwiners conceals a potential danger, as it may
happen that the boundary process degenerates to a deterministic process.
Therefore, if one is interested in constructing interesting
infinite-dimensional Markov processes (as we do), one needs additional
arguments guaranteeing that such a degeneration does not occur. We were
fortunate that we were able to explicitly compute the generator of our process
$X^\z$: from the fact that the generator is a second order operator it is easy
to conclude that $X^\z$ cannot be a deterministic process.

Finally, note that the existence of a nontrivial stationary distribution,
$M^\z$, makes it possible to prove the non-determinism of the boundary process
in a different way.

\end{document}